\documentclass[11pt]{amsart}

\usepackage{enumerate,url,amssymb,  mathrsfs, graphicx, pdfsync}
\newtheorem{theorem}{Theorem}[section]

\newtheorem*{lemma*}{Lemma}

\theoremstyle{definition}

\newtheorem{remark}[theorem]{Remark}
\usepackage[justification=centering]{caption}

\theoremstyle{remark}

\numberwithin{equation}{section}

\newcommand\numberthis{\addtocounter{equation}{1}\tag{\theequation}}

\renewcommand{\bar}[1]{\overline{#1}}

\newcommand{\yy}{\mathbb{Y}}

\newcommand{\dd}{\mathbb{D}}

\newcommand{\abs}[1]{\lvert#1\rvert}

\newcommand{\C}{\mathbb{C}}

\newcommand{\W}{\mathscr{W}}

\renewcommand{\L}{\mathscr{L}}

\newcommand{\I}{\mathcal{I}}
\newcommand{\M}{\mathcal{M}}

\newcommand{\R}{\mathbb{R}}
\newcommand{\X}{\mathbb{X}}

\newcommand{\Y}{\mathbb{Y}}

\newcommand{\dtext}{\textnormal d}

\newcommand{\onto}{\xrightarrow[]{{}_{\!\!\textnormal{onto\,\,}\!\!}}}

\def\leq{\leqslant}
\def\geq{\geqslant}

\def\le{\leqslant}

\def\Xint#1{\mathchoice
{\XXint\displaystyle\textstyle{#1}}%
{\XXint\textstyle\scriptstyle{#1}}%
{\XXint\scriptstyle\scriptscriptstyle{#1}}%
{\XXint\scriptscriptstyle\scriptscriptstyle{#1}}%
\!\int}
\def\XXint#1#2#3{{\setbox0=\hbox{$#1{#2#3}{\int}$}\vcenter{\hbox{$#2#3$}}\kern-.5\wd0}}
\def\dashint{\Xint-}

\def\XXiint#1#2#3{{\setbox0=\hbox{$#1{#2#3}{\iint}$}\vcenter{\hbox{$#2#3$}}\kern-.5\wd0}}

\begin{document}
\title[Sobolev Homeomorphic Extensions]{Characterizing Sobolev Homeomorphic Extensions via Internal Distances}

\author[A. Koski]{Aleksis Koski}
\address{Department of Mathematics and Systems Analysis, Aalto University, FI-00076 Aalto, Finland}
\email{aleksis.koski@aalto.fi}

\author[J. Onninen]{Jani Onninen}
\address{Department of Mathematics, Syracuse University, Syracuse,
NY 13244, USA and  Department of Mathematics and Statistics, P.O.Box 35 (MaD) FI-40014 University of Jyv\"askyl\"a, Finland
}
\email{jkonnine@syr.edu}

\author[H. Xu]{Haiqing Xu}
\address{Frontiers Science Center for Nonlinear Expectations (Ministry of Education of China), Research Center for Mathematics and Interdisciplinary Sciences, Shandong University, Qingdao 266237, China}
\email{haiqing.xu@email.sdu.edu.cn}

\thanks{A. Koski was supported by the Academy of Finland grant number 355840.
J. Onninen was supported by the NSF grant  DMS-2154943.   
H. Xu was supported by the National Natural Science Foundation of China (No.~12201349)}

\subjclass[2010]{Primary 46E35. Secondary 42B25, 58E20}


\keywords{Sobolev homeomorphisms, Sobolev extensions, Jordan-Sch\"onflies theorem, Gagliardo's trace theorem}

\begin{abstract} 
We give a full  characterization of embeddings of the unit circle that admit a Sobolev homeomorphic extension to the unit disk. As a direct corollary, 
we establish that for quasiconvex target domains $\Y$, any homeomorphism $\varphi \colon \partial \mathbb{D} \onto \partial \Y$ that admits a continuous $\W^{1,p}$-extension to the unit disk $\mathbb{D}$ also admits a $\W^{1,p}$-homeomorphic extension. These Sobolev variants of the classical Jordan-Sch\"onflies theorem are essential for ensuring  the well-posedness of variational problems arising in Nonlinear Elasticity and Geometric Function Theory.
\end{abstract}

\maketitle
\section{Introduction}
Let $\mathbb S $ denote the planar unit circle, and let $\varphi \colon  \mathbb S \to \mathbb{C} $ be a topological embedding. Two
classical extension theorems apply: The {\it Jordan--Sch\"onflies theorem} asserts that there exists a homeomorphism $h \colon  \mathbb{C} \to \mathbb{C}$ such that $ h $ agrees with $\varphi$ on  $ \mathbb S$. In particular, the set $\varphi(\mathbb S)$ separates the plane into two domains: one bounded and the other unbounded. 

Throughout this text,  $\mathbb Y \subset \mathbb{C}$ denotes a bounded Jordan domain, and $\mathbb D$  the unit disk. Another classical extension theorem asserts that a homeomorphism  $\varphi \colon \partial \mathbb D  \onto \partial \mathbb Y$ admits a continuous extension to $\overline{ \mathbb D }$ in the Sobolev class $\W^{1,p} (\mathbb D, \mathbb  C)$, $1<p<\infty$, if and only if it satisfies the so-called {\it $p$-Douglas condition},
\begin{equation}\label{eq:pdouglas}
\int_{\partial \mathbb D} \int_{\partial \mathbb D}  \frac{\abs{\varphi (x) - \varphi (y)}^p}{\abs{x-y}^p} \dtext x \dtext y < \infty \, , \qquad 1<p< \infty \, . 
\end{equation}
This condition encapsulates the regularity of $\varphi$ through a delicate interplay between boundary behavior and Sobolev smoothness. Its origins trace back to the classical  work of Douglas~\cite{Do} for $p=2$. In full generality, this result is part of {\it Gagliardo’s trace theorem}~\cite{Ga}, a cornerstone in the theory of boundary value problems, characterizing the boundary data that permit Sobolev extensions into the domain. This theorem plays a fundamental role in PDEs, geometric analysis, and variational problems. If $p=1$,   {Gagliardo's theorem} states that  a given homeomorphism $\varphi \colon \partial \mathbb D \onto \partial \Y$ has a Sobolev extension to $\mathbb D$ in $\W^{1,1} (\mathbb D , \mathbb C)$ exactly when $\varphi \in \L^1(\partial \mathbb D)$. The case $p=\infty$, on the other hand, follows from the classical {\it Kirszbraun extension theorem}~\cite{Ki} which says that a mapping $\varphi \colon \partial \mathbb D \to \C$ has a Lipschitz extension to $\overline{\mathbb D}$ if and only if $\varphi$ is Lipschitz regular.

A natural and fundamental question arises,  motivated by the well-posedness of homeomorphic variational problems in Geometric Function Theory~\cite{AIMb, IMb, IMOsur2,   Reb} and Nonlinear Elasticity~\cite{Anb, Bac, Ba2, Cib}:  \emph{Can one construct an extension that satisfies both the Jordan–Schönflies theorem and the Sobolev condition?}
In general, the answer to this Sobolev variant of the Jordan–Sch\"onflies theorem is fundamentally negative.

The only case where a positive result is guaranteed is when $p=\infty$: every Lipschitz homeomorphism $\varphi \colon \partial \mathbb D \onto \partial \Y$ admits a Lipschitz homeomorphic extension to $\mathbb D$,~\cite{Kovalev, HKOext}. For each $p < \infty $, however, there exist Jordan domains $\mathbb{Y}$ and boundary homeomorphisms $\varphi \colon \partial \mathbb{D} \to \partial \mathbb{Y} $ that satisfy the $p$-Douglas condition but fail to admit even a $\W^{1,1}$-homeomorphic extension to the disk~\cite{Zh, KOext3}. The core difficulty lies in the absence of general methods for constructing Sobolev homeomorphisms, reflecting a significant gap in the literature.

Nevertheless,  the main result of this paper establishes a homeomorphic counterpart of the $p$-Douglas condition~\eqref{eq:pdouglas}, providing a complete characterization of boundary homeomorphisms   $\varphi \colon \partial \mathbb D \onto \partial \Y$ that admit a Sobolev homeomorphic extension to the unit disk.

\begin{theorem}\label{thm:main}
For $1< p < \infty $, a  homeomorphism $\varphi \colon \partial \mathbb D \onto \partial \mathbb Y$ admits a homeomorphic extension $h \colon \overline{\mathbb D} \onto \overline{\mathbb Y}$ in the Sobolev class $\W^{1,p} (\mathbb D, \mathbb C)$ if and only if it satisfies the \emph{internal $p$-Douglas condition}; that is, 
\begin{equation}\label{eq:internalDouglas}
\int_{\partial \dd}\int_{\partial \dd} \frac{[d_{\yy}(\varphi(x),\varphi(y)]^p}{|x-y|^p} \, dx \, dy \ < \ \infty.
\end{equation}
\end{theorem}
Here the \emph{internal distance} $d_{\yy} : \bar{\yy} \times \bar{\yy} \to [0,\infty)$ is the metric on $\bar{\yy}$ defined by
\[d_{\yy}(x,y) = \inf_\gamma |\gamma|,\]
where $|\gamma|$ denotes the length of $\gamma$, and the infimum ranges over all rectifiable curves $\gamma \subset \bar{\yy}$ which connect $x$ to $y$.

The classical Gagliardo's trace theorem characterizes traces of Sobolev spaces when the reference configuration is a Lipschitz domain. Since the Sobolev spaces under consideration  and the internal $p$-Douglas condition~\eqref{eq:internalDouglas} remain invariant under a global bi-Lipschitz change of variables in  the reference domains, Theorem~\ref{thm:main},  the homeomorphic counterpart of Gagliardo's trace theorem,  immediately extends to the case of Lipschitz domains (an axiomatic assumption in the theory of Nonlinear Elasticity).

\begin{theorem}
Let $\X$ and $\Y$ be simply connected bounded Lipschitz domains in the complex plane and  $1<p<\infty$. A  boundary homeomorphism $\varphi \colon \partial \mathbb X \onto \partial \mathbb Y$ admits a homeomorphic extension $h \colon \overline{\mathbb X} \onto \overline{\mathbb Y}$ in the Sobolev class $\W^{1,p} (\mathbb X, \mathbb C)$ if and only if 
\[
\int_{\partial \X}\int_{\partial \X} \frac{[d_{\yy}(\varphi(x),\varphi(y)]^p}{|x-y|^p} \, dx \, dy \ < \ \infty.
\]
\end{theorem}

Similar to the classical trace theorem, the case $p = 1$ exhibits distinct behavior. Although the condition~\eqref{eq:internalDouglas} with $p = 1$ ensures the existence of a $\W^{1,1}$-homeomorphic extension to the unit disk $\mathbb{D}$ (see Section~\ref{sec:sufficiency}), it is not a necessary condition.  

\begin{theorem}\label{thm:p=1case}  
 There exists a homeomorphism $h \colon \bar{\mathbb{D}} \to \bar{\mathbb{Y}}$ belonging to $\W^{1,1}(\mathbb{D}, \mathbb{C}) $ whose restriction to the boundary does not satisfy the internal Douglas condition \eqref{eq:internalDouglas} for $p = 1$.  
\end{theorem}


Theorem~\ref{thm:main} naturally leads to a follow-up question: For which Jordan domains $\Y$ are the conditions~\eqref{eq:pdouglas} and~\eqref{eq:internalDouglas} equivalent? In other words,  for which Jordan domains $\Y$ does every homeomorphism $\varphi \colon \partial \mathbb{D} \onto \partial \Y$ that admits a continuous $\W^{1,p}$-extension to $\mathbb{D}$ also admit a $\W^{1,p}$-homeomorphic extension?    

For Lipschitz regular targets $\Y$, this equivalence holds. Indeed,   in such cases, the problem reduces to  $\Y=\mathbb D$ by applying a bilipschitz change of variables.  The equivalence is then immediate, and analytic tools such as the  Rad\'o-Kneser-Choquet (RKC) theorem~\cite{Dub}  provide a direct method to construct the desired homeomorphic extension.  Specifically, the RKC theorem guarantees that any boundary homeomorphism $\varphi \colon \partial \mathbb D \onto \partial \mathbb D$ admits a homeomorphic harmonic extension to $\mathbb D$. This harmonic extension belongs to $\W^{1,p}(\mathbb{D}, \R^2)$  if and only if $\varphi$ satisfies the $p$-Douglas condition~\cite[Theorem 1.2]{KObiSobo}. Moreover, it always lies in  $\W^{1,p}(\mathbb{D}, \R^2)$ for every $p<2$, see~\cite{Ve}.

Earlier attempts to address this question highlight the complexity of the problem, particularly for Jordan domains with rectifiable boundaries. Consider first the case where the target domain $\Y$ has a piecewise smooth boundary, $p \neq 2 $, and the boundary mapping $ \varphi \colon \partial \mathbb{D} \onto \partial \Y $ admits a $\W^{1,p} $-Sobolev extension to  $\mathbb{D} $. In this case, it is known that $ \varphi $ also admits a homeomorphic extension to $ \mathbb{D} $ in the Sobolev space $\W^{1,p} (\mathbb{D}, \mathbb{R}^2)$; see~\cite{KOext3}. However, this result fails to generalize to the case $p = 2$, due to geometric obstructions such as cusps. 

 Second, for a Jordan domain $\Y$  with just a rectifiable boundary, any homeomorphism $\varphi \colon \partial \mathbb D \onto \partial \mathbb Y$ extends as a homeomorphism $h \colon \overline{ \mathbb D} \onto \overline{ \Y}$ in the Sobolev class $\W^{1,p} (\mathbb D, \mathbb R^2)$ for all $p<2$,~\cite{KOext1}. For $2 < p < \infty$ and a rectifiable target domain, it has been conjectured that the $p$-Douglas condition \eqref{eq:pdouglas} would also guarantee the existence of such an extension. However, we show that this conjecture is false.
 \begin{theorem}\label{thm:rectifiable_p>2}
Let $p\in (2, \infty)$. There is a Jordan domain $\Y$ with rectifiable boundary together with a homeomorphism $\varphi \colon \partial \mathbb D \onto \partial \Y$ that satisfies the $p$-Douglas condition but it does not admit a homeomorphic extension to $\mathbb D$ in $\W^{1,p} (\mathbb D, \R^2)$.
\end{theorem}

To summarize the equivalences of the conditions~\eqref{eq:pdouglas} and~\eqref{eq:internalDouglas}:

 \vskip0.3cm
 \begin{center}
 
  \begin{tabular}{ |p{3.0cm}||p{1.8cm}|p{1.3cm}|p{1.8cm}|p{1.3cm}|  }
 \hline
$\partial \Y$& $1\le p <2$ & $p=2$ & $2<p<\infty$ & $p=\infty$\\
 \hline
Arbitrary     & No    &No&   No& Yes \\
Rectifiable  &   Yes   & No   & No&  Yes \\
Piecewise smooth& Yes   & No&Yes  &  Yes  \\
Lipschitz graph   &  Yes   &  Yes &    Yes  &  Yes  \\
\hline
\end{tabular}
 \end{center}
\vskip0.3cm
In particular, somewhat surprisingly, when $p > 2$ and $\Y $ has a piecewise smooth boundary, a homeomorphism $\varphi \colon \partial \mathbb{D} \onto \partial \Y $ satisfies the $p$-Douglas condition if and only if it satisfies the internal $p$-Douglas condition. This equivalence holds even though the internal distances in $\Y$ may not be comparable to their Euclidean distances, see Remark~\ref{rm:fail_proof}.

In contrast, for $p = 2$, this equivalence breaks down, even for an inner cusp domain with any power-type cusp on the boundary~\cite{KOext3}. In this case, one may reestablish the equivalence by assuming also that the target domain is quasiconvex. Recall that a Jordan domain $\yy $ is called \emph{quasiconvex} if there exists a constant $C$ such that, for every pair of points $x, y \in \bar{\yy}$, the internal distance satisfies  
\[
d_{\yy}(x, y) \leq C |x - y|.
\]  
Quasiconvexity plays a fundamental role in Geometric Function Theory (GFT)~\cite{GeMo, Ge2, NV, Va1.5}. In particular, an important subclass of quasiconvex domains is formed by \emph{quasidisks}, which arise as images of the unit disk under quasiconformal self-maps of $\mathbb{C}$. The boundary of a quasidisk need not be rectifiable.

In the case of a quasidisk target, it is already known that the $p$-Douglas condition is sufficient to guarantee a $\W^{1,p}$-Sobolev homeomorphic extension~\cite{KKO, KOext3}. However, this fact also follows directly from the following more general result, which is an immediate consequence of Theorem~\ref{thm:main}.  

%
 
%

\begin{theorem}\label{thm:quasiconvex}  
Let $\yy $ be a quasiconvex Jordan domain, and let $\varphi \colon \partial \dd \onto \partial \yy $ be a homeomorphic boundary map. If $\varphi $ satisfies the $ p$-Douglas condition \eqref{eq:pdouglas}, then it admits a homeomorphic extension $ h \in \W^{1,p}(\dd, \R^2)$.  
\end{theorem}

\textbf{Outline of the paper.} The proof of our main result, Theorem \ref{thm:main}, is divided into two sections covering the sufficiency and necessity of the condition \eqref{eq:internalDouglas}. The sufficiency part is based on a new geometric extension method, which utilizes a novel framework of constructing a dyadic system of hyperbolic geodesics within the target domain, giving a decomposition of the target via piecewise smooth regions not touching the boundary. The construction allows for increased flexibility compared to previous extension methods, and enables us to forego any additional geometric assumptions such as those required in prior works around this problem \cite{KKO,KOext1,KOext3}.  

Proving that the internal $p$-Douglas condition~\eqref{eq:internalDouglas} is necessary for the existence of a homeomorphic extension also presents significant challenges. Our proof draws heavily on the spherical maximal inequality, as developed by Bourgain and Stein in~\cite{Bo, St_sp}. To the best of our knowledge, this represents the first successful application of spherical maximal functions to trace problems. The proof is then completed using the classical theory of the associated maximal operators.

\section{Proof of Theorem \ref{thm:main}, sufficiency}\label{sec:sufficiency}
In this section, we prove the sufficiency part of our main result, Theorem \ref{thm:main}. In fact, the proof that follows also extends to the case $p = 1$, showing that condition \eqref{eq:internalDouglas} is sufficient to guarantee the existence of a homeomorphic $\W^{1,p}$-extension for all $p \geq 1$. 

\begin{proof} Suppose that a boundary homeomorphism $\varphi : \partial \dd \to \partial \yy$ satisfies \eqref{eq:internalDouglas}. The strategy will be to find appropriate dyadic decompositions on both the domain and target side and use them to define a homeomorphic Sobolev extension $h : \dd \to \yy$ piece by piece.

In this and the later parts of the paper, we will employ a fixed dyadic decomposition of the unit circle $\partial \dd$ into dyadic arcs. These dyadic arcs will always be denoted by $\{I_{n,k}\}$, with $n = 0,1,\ldots$ and $k = 1,\ldots,2^n$, with each arc $I_{n,k}$ having length comparable to $2^{-n}$ respectively. 

At some parts of the proof, for the ease of presentation, we will opt to pretend that the boundary $\partial \dd$ is locally flat, which is why the dyadic arcs $I_{n,k}$ will simply be referred to as dyadic intervals. If needed, this convention can be made rigorous by considering a local bilipschitz change of variables into the upper half space. This convention is particularly reflected in how we refer to orientation, choosing to call the endpoints of each $I_{n,k}$ the left and right endpoint for example. We may also suppose that the boundary map $\varphi : \partial \dd \to \partial \yy$ is positively oriented. 

We would next like to construct certain crosscuts inside the target domain $\yy$, but before we can do so we need to consider certain subintervals of the $I_{n,k}$. We first divide each $I_{n,k}$ into eight intervals of equal length, denoted by $I^j_{n,k}$, $j = 1,\ldots,8$ from left to right. We call these the pieces of $I_{n,k}$. For each dyadic interval $I_{n,k}$, we now wish to define a specific crosscut $C_{n,k}$ in $\yy$ corresponding to it. The construction will be slightly different in the case where $n$ is an even number versus when $n$ is odd, so let us explain the case where $n$ is even first.

Let us pick $A_{n,k} = I^1_{n,k}$ as the leftmost piece of $I_{n,k}$. Moreover, we pick another interval $B_{n,k} = I^2_{n,k+1}$ as the second piece of the neighbouring dyadic interval $I_{n,k+1}$.

Let us now consider the integral expressions
\begin{equation}\label{eq:oneterm}
\I_{n,k} = \int_{A_{n,k}}\int_{B_{n,k}} \frac{d_{\yy}(\varphi(x),\varphi(y))^p}{|x-y|^p} \, dx \, dy.
\end{equation}
The key thing to note about the $\I_{n,k}$ is that the domains of integration $A_{n,k} \times B_{n,k}$ are in fact mutually disjoint subsets of $\partial \dd \times \partial \dd$ for each distinct pair $(n,k)$. Hence the sum $\sum_{n=0}^\infty \sum_{k=1}^{2^n} \I_{n,k}$ converges by \eqref{eq:internalDouglas}.

The next thing we note is that the expression $|x-y|^{-p}$ within this domain of integration is comparable to $(2^{-n})^{-p}$, which will help with the following estimate. The estimate is that by a pointwise inequality, there must be at least one pair of points $a_{n,k} \in A_{n,k}$ and $b_{n,k} \in B_{n,k}$ such that
\[d_{\yy}(\varphi(a_{n,k}),\varphi(b_{n,k}))^p (2^{-n})^{2-p} \leq c_p \I_{n,k},\]
where $C_p$ only depends on $p$. We may now define the crosscut $C_{n,k}$ in $\yy$ as the hyperbolic geodesic between the boundary points $\varphi(a_{n,k})$ and $\varphi(b_{n,k})$. The Gehring-Hayman theorem says that the length of this curve is comparable to the internal distance between its endpoints, which lets us obtain the following key estimate
\begin{equation}\label{eq:partEstimate}
\left(2^{-n}\right)^{2-p} |C_{n,k}|^p  \leq c_p \I_{n,k}.
\end{equation}
Eventually, we will split the domain side into dyadic regions $U_{n,k}$ which lie ``above" their corresponding interval $I_{n,k}$, and if we can guarantee that the extension $h$ maps each such region to a subset $V_{n,k}$ of $\yy$ with perimeter essentially comparable to $|C_{n,k}|$, then with \eqref{eq:partEstimate} we will be able to show that $h$ lies in the correct Sobolev class. The obstruction now is purely topological, as we will need to use the crosscuts $C_{n,k}$ to decompose the target side $\yy$ and construct the dyadic regions $V_{n,k}$ which will let us define the extension $h$.

Regarding the structure of the proof, we will start by constructing the extension near the boundary $\partial \dd$. This simply amounts to disregarding what happens at the dyadic levels for small values of $n$. From now on, let us assume that $n \geq 4$ so that at the initial value $n = 4$ there are already $16$ crosscuts $C_{n,k}$. At the end of the proof, we will define the extension in the remaining ``central region" that will be left over.

\textbf{Crossing of curves.} Let us begin by looking at some topological properties of the curves $C_{n,k}$. We are particularly interested in characterizing the instances in which these curves intersect inside of $\yy$. Any conformal map $g : \dd \to \yy$ preserves hyperbolic geodesics, so we simply need to consider when two hyperbolic geodesics intersect in $\dd$. In this case the geodesics are circular arcs orthogonal to the boundary, so the answer simply depends on the ordering of their endpoints with respect to the orientation of the boundary, see Figure \ref{fig:QHdisc}. What this means for us is that two of our crosscuts $C_{n,k}$ and $C_{n',k'}$ will intersect if and only if their endpoints are in alternating order on $\partial \yy$, and in this case the curves intersect at exactly one point.

\begin{figure}
\centering
\includegraphics[scale=0.25]{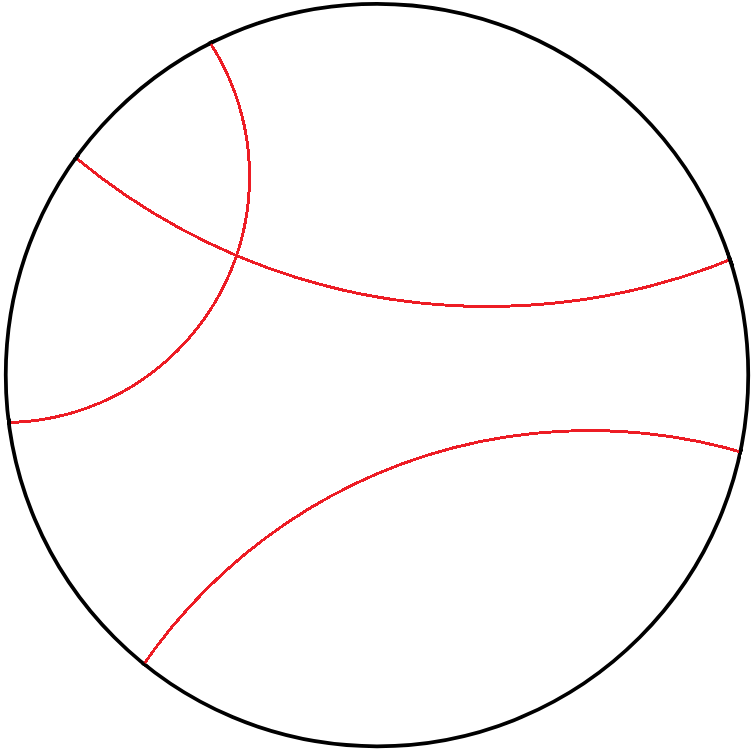}
\caption{The intersection of quasihyperbolic geodesics in $\dd$ is based on the order of endpoints.}
\label{fig:QHdisc}
\end{figure}

Note now that for each $n,k$ the curves $C_{n,k-1}$ and $C_{n,k}$ will necessarily intersect each other since their endpoints are positioned alternatingly (the order is $\varphi(a_{n,k-1}),\varphi(a_{n,k}),\varphi(b_{n,k-1}),\varphi(b_{n,k})$). We wish to get rid of the curves $C_{n,k}$ in order to work with a set of curves with less intersection points. For this purpose, let $\Gamma_{n,k}$ denote the hyperbolic geodesic from $\varphi(a_{n,k})$ to $\varphi(a_{n,k+1})$. We claim that
\begin{equation}\label{eq:newcurve}
|\Gamma_{n,k}| \leq C(|C_{n,k}| + |C_{n,k+1}|), 
\end{equation}
where $C$ is universal. This is simply due to the fact that one may travel from $\varphi(a_{n,k})$ to $\varphi(a_{n,k+1})$ going along the curves $C_{n,k}$ and $C_{n,k+1}$ respectively, which establishes that the internal distance from $\varphi(a_{n,k})$ to $\varphi(a_{n,k+1})$ is controlled by the lengths of these two curves. The rest is another application of the Gehring-Hayman theorem, since $\Gamma_{n,k}$ is also a hyperbolic geodesic.

\begin{figure}
\centering
\includegraphics[scale=0.25]{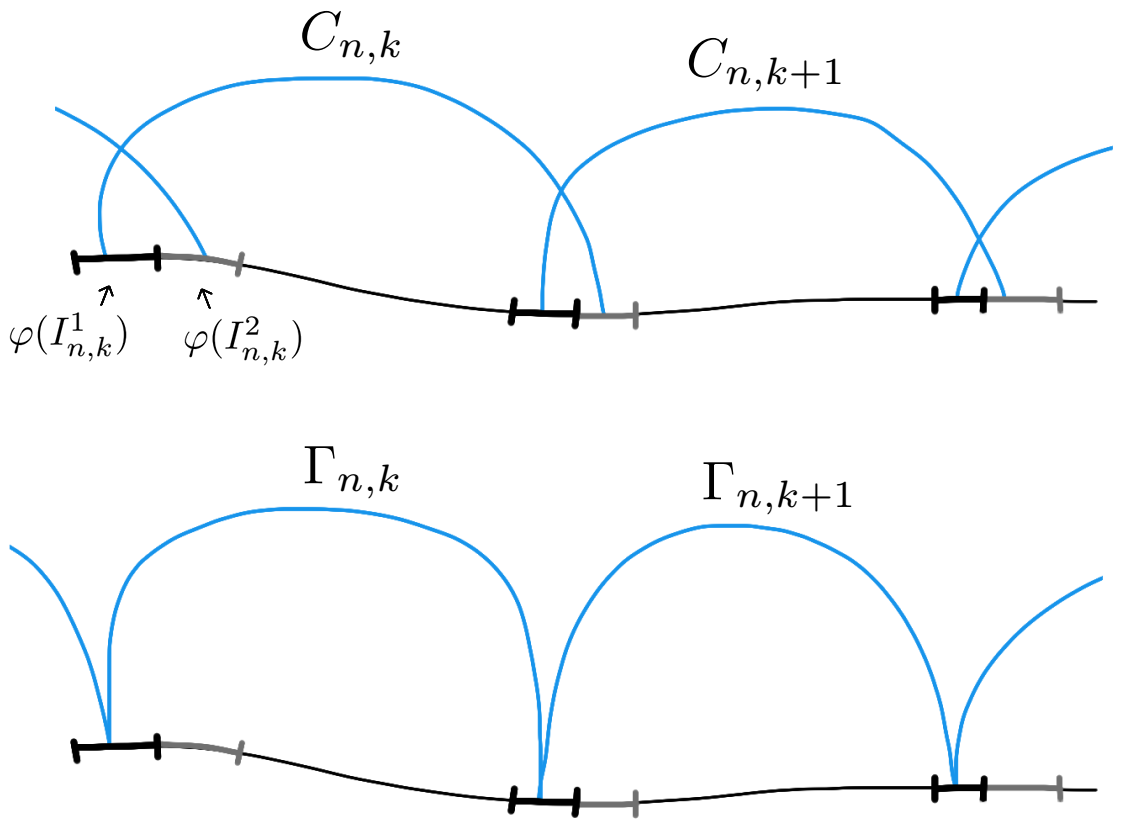}
\caption{Replacing the curves $C_{n,k}$ with the curves $\Gamma_{n,k}$.}
\label{fig:startcurves}
\end{figure}

Via this replacement, we have obtained new crosscuts $\Gamma_{n,k}$ which now have endpoints $\varphi(a_{n,k})$ to $\varphi(a_{n,k+1})$ instead. The advantage gained is that for a fixed $n$ the $\Gamma_{n,k}$ do not cross each other, and only share an endpoint with each of their respective neighbours. Combining \eqref{eq:newcurve} and \eqref{eq:partEstimate} gives
\begin{equation}\label{eq:newcurveT}
\left(2^{-n}\right)^{2-p} |\Gamma_{n,k}|^p \leq c_p(\I_{n,k} + \I_{n,k+1}).
\end{equation}

Unfortunately, we must further complicate this construction by adding another analogous set of curves $\Gamma^{*}_{n,k} \subset \yy$ on the same dyadic level. The curves $\Gamma^{*}_{n,k}$ will be defined in exactly the same way as the curves $\Gamma_{n,k}$, but using different starting intervals in place of $A_{n,k}$ and $B_{n,k}$. The intervals we will use instead are $A^{*}_{n,k} = I^5_{n,k}$ and $B^*_{n,k} = I^6_{n,k+1}$, so we essentially just shift the construction by half a length of $I_{n,k}$ to the right. With similar reasoning as was used to obtain \eqref{eq:newcurveT}, we may obtain the analogous estimate
\begin{equation}\label{eq:newcurveT1}
\left(2^{-n}\right)^{2-p} |\Gamma_{n,k}^*|^p \leq c_p(\I_{n,k+1} + \I_{n,k+2}).
\end{equation}

\begin{figure}
\centering
\makebox[\textwidth][c]{%
    \includegraphics[scale=0.26]{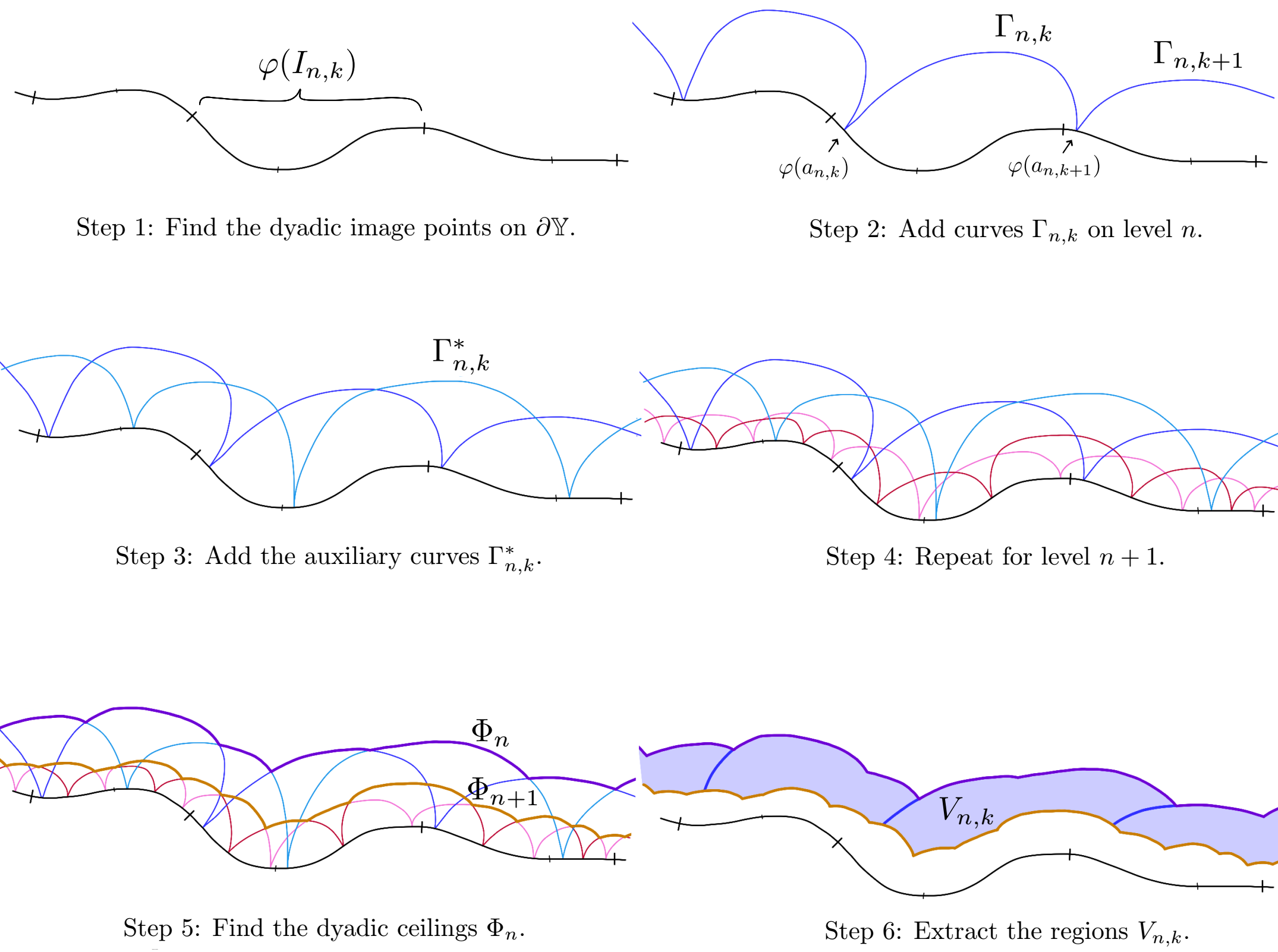}
}
\caption{A step-by-step view of the entire construction on the image side.}
\label{fig:construction}
\end{figure}

We can then repeat the construction as before and find curves $\Gamma^{*}_{n,k}$ starting from a point $\varphi(a^*_{n,k})$ and ending up at the point $\varphi(a^*_{n,k+1})$, where $a^*_{n,k}$ is a point on $I^3_{n,k}$ for each $k$. See Step 2 and 3 in Figure \ref{fig:construction}. 

\textbf{The case when $n$ is odd.} We next explain how this construction changes when $n$ is odd. The purpose of this change is simply to control the amount of intersection points between successive generations of the crosscuts $\Gamma_{n,k}$ and $\Gamma^{*}_{n,k}$. To achieve this, we have to take a mirror image of the construction from the case where $n$ was even. What we mean by this is that in the case where $n$ is odd we will define the initial crosscuts $C_{n,k}$ using the piece intervals $A_{n,k} = I^8_{n,k}$ and $B_{n,k} = I^7_{n,k-1}$ instead of what we chose previously. Note that the order of the intervals is flipped compared to the case where $n$ was even, but otherwise the construction may proceed as before.

For $n$ odd, this then lets us define the curves $\Gamma_{n,k}$ as curves in $\yy$ between $\varphi(a_{n,k})$ and $\varphi(a_{n,k+1})$, where $a_{n,k}$ is always a point on $I^8_{n,k-1}$. So in a practical sense the only thing which has changed from the case of even $n$ is that now the points $a_{n,k}$ are on the left hand side of the left endpoint of $I_{n,k}$ instead of on the right hand side.

The curves $\Gamma^{*}_{n,k}$ for $n$ odd are similarly defined as curves between $\varphi(a^*_{n,k})$ and $\varphi(a^*_{n,k+1})$, where $a^*_{n,k}$ is a point on $I^4_{n,k}$ for each $k$.

Let us now study the intersection points between curves on the same and successive generations. On a fixed generation $n$, the curves $\Gamma_{n,k}$ and $\Gamma^{*}_{n,k}$ must intersect once due to the ordering of their endpoints. Similarly, the curves $\Gamma^{*}_{n,k}$ and $\Gamma_{n,k+1}$ must also intersect once. Since $\Gamma_{n,k}$ and $\Gamma_{n,k+1}$ do not intersect, their two intersection points with $\Gamma^{*}_{n,k}$ must be in order of $\Gamma_{n,k}$ first, $\Gamma_{n,k+1}$ last, when traversing $\Gamma^{*}_{n,k}$ from $\varphi(a^*_{n,k})$ to $\varphi(a^*_{n,k+1})$. The same goes for how $\Gamma_{n,k}$ intersects the two curves $\Gamma^{*}_{n,k-1}$ and $\Gamma^{*}_{n,k}$. So the intersection points are as depicted in Step 3 in Figure \ref{fig:construction}. 

We now note that on the next dyadic level $n+1$, for each crosscut of the form $\Gamma_{n+1,k}$ or $\Gamma^{*}_{n+1,k}$ there is a curve on the previous level, either $\Gamma_{n,j}$ or $\Gamma^{*}_{n,j}$ for an appropriate $j$, such that the curve picked on level $n+1$ does not intersect the chosen curve on level $n$. This is due to the ordering of the endpoints, as our construction ensures that the two endpoints of each crosscut always lie between the endpoints of a curve on the previous level, see Step 4 in Figure \ref{fig:construction}. We call this the \emph{successive non-crossing property}, as we will need to refer to it later.

\textbf{Dyadic ceilings.} Let us next define, for each dyadic level $n$, a Jordan curve $\Phi_n$ which lies inside of $\yy$. This curve is defined as follows. We take any curve $\Gamma_{n,k}$, and starting from its left endpoint $\varphi(a_{n,k})$ we follow along the curve until we hit its intersection point with $\Gamma^*_{n,k}$. Take this as a starting point of $\Phi_n$. We follow along $\Gamma^*_{n,k}$ with the same left to right orientation until we hit its intersection point with $\Gamma_{n,k+1}$.

In this way we alternatingly travel between the two sets of curves on the same dyadic level $n$, and eventually come back to the starting point on $\Gamma_{n,k}$. This closes off a Jordan curve inside $\yy$ defined as $\Phi_n$, see Step 5 in Figure \ref{fig:construction}. The curve $\Phi_n$ can be regarded as a ``dyadic ceiling" at level $n$. An alternative way to define $\Phi_n$ is simply to use the crosscuts $\Gamma_{n,k}$ and $\Gamma_{n,k}^*$ (with fixed $n$ and each $k$) to cut off parts from the domain $\yy$, and then defining $\Phi_n$ as the boundary of the region which is left. 
\\\\
\textbf{Claim.} The curves $\Phi_n$ and $\Phi_{n+1}$ do not intersect.

\emph{Proof of claim.} This is a direct consequence of the successive non-crossing property, as this property guarantees that no curve of the form $\Gamma_{n+1,k}$ or $\Gamma_{n+1,k}^*$ intersects the curve $\Phi_n$. In simple terms, this means that the dyadic ceiling at level $n$ is always ``higher" than the one at level $n+1$.
\\\\
Let us next use the curves we are working with to construct, for each $n$ and $k$, a region $V_{n,k} \subset \yy$ inside the target domain. These will provide a way to split the target domain dyadically into pieces.

The region $V_{n,k}$ will be bounded by a Jordan curve consisting of four parts. Two of these parts will be subcurves of $\Phi_n$ and $\Phi_{n+1}$, so it remains to pick the remaining two parts which will naturally be curves connecting these two dyadic ceilings. The two curves we pick for this purpose will simply be appropriate subcurves of $\Gamma_{n,k}$ and $\Gamma_{n,k+1}$. This restriction already uniquely defines the region $V_{n,k}$, but we give an explicit explanation as well.

Start from the endpoint $\varphi(a_{n,k})$ and follow along $\Gamma_{n,k}$ until we reach a point on $\Phi_{n+1}$ which we call $v_{n,k}^-$. This is where the boundary curve $\partial V_{n,k}$ starts from. Continue along $\Gamma_{n,k}$ until we reach a point on $\Phi_n$, which we call $v_{n,k}^+$. We continue from $v_{n,k}^+$ by following along $\Phi_n$ until we hit the point $v_{n,k+1}^+$, which is a point defined analogously to $v_{n,k}^+$ but using $\Gamma_{n,k+1}$ instead of $\Gamma_{n,k}$. We follow backwards from $v_{n,k+1}^+$, retracing the path that was taken from $\varphi(a_{n,k+1})$ to $v_{n,k+1}^+$ along $\Gamma_{n,k+1}$. We eventually reach $\Phi_{n+1}$ at a point we call $v_{n,k+1}^-$, and then follow backwards along $\Phi_{n+1}$ until we reach the initial starting point $v_{n,k}^-$. This gives a closed Jordan curve which then bounds the region $V_{n,k}$. See Step 6 in Figure \ref{fig:construction}, and Figure \ref{fig:horrible}.

\begin{figure}
\centering
\includegraphics[scale=0.29]{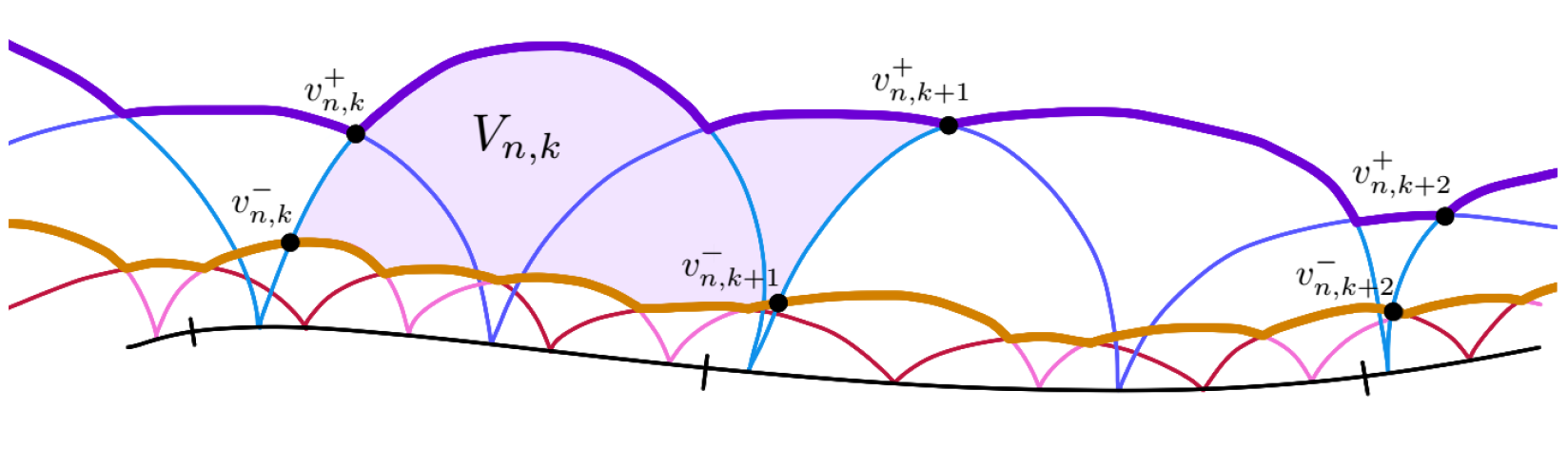}
\caption{The construction of the regions $V_{n,k}$ on the target side.}
\label{fig:horrible}
\end{figure}

Let us next define, for each $V_{n,k}$, a corresponding preimage region $U_{n,k} \subset \dd$ on the domain side. We start by defining, for each dyadic level $n$, the curve $\phi_n$ as the circle $\{(1-2^{-n}) e^{i\theta} : \theta \in [0,2\pi)\}$. The boundary of the region $U_{n,k}$ will again consist of four parts, two of which are subcurves of $\phi_n$ and $\phi_{n+1}$. In fact, we will pick two points on $\phi_n$, two points on $\phi_{n+1}$, and then connect these pairwise via linear segments to form $U_{n,k}$ as a sort of quadrilateral with two sides being circular arcs.

For presentation purposes, we again imagine the boundary $\partial \dd$ as being locally flat. In this case the curve $\phi_n$ can be interpreted as simply a horizontal line at height $2^{-n}$ from the boundary (though it is actually a circle). We now wish to choose the four vertices of the quadrilateral $U_{n,k}$. The only complication in this part of the construction is to respect the topology of the image regions, mainly in how the regions $V_{n,k}$ are aligned between successive levels. We elaborate on this a bit.

Let us consider a single target side region $V_{n,k}$, and we call the \emph{bottom side} of this region the part of $\partial V_{n,k}$ which lies on $\Phi_{n+1}$. The main concern here is how the top vertices of the sets $V_{n+1,j}$ on the next dyadic level are positioned, and by top vertices we mean the appropriate points $v^+_{n+1,j}$. There will  be anywhere from one to three of such top vertices positioned on the bottom side of a fixed region $V_{n,k}$. In fact, there are three curves $\Gamma_{n+1,j}$ which have at least one endpoint between the endpoints of $\Gamma_{n,k}$. These give rise to three top vertices on level $n+1$ which have the potential to be on this bottom side. The one in the middle is guaranteed to be on the bottom side due to topology (see e.g. Figure \ref{fig:horrible}), and we call such a situation a \emph{top point in the middle}. The two others can be either on the bottom side of $V_{n,k}$ or either of its neighbours $V_{n,k-1}$ and $V_{n,k+1}$ respectively. Hence for every vertex $v_{n,k}^-$ on the bottom side on level $n$, there is one particular top vertex on level $n+1$ which can a priori be on either side of this point on the curve $\Phi_{n+1}$. It could also be equal to $v_{n,k}^-$. Let's call these three options \emph{left, right,} and \emph{equal}. Whichever option here happens will not pose a problem for the construction, but it will have to be respected in the construction of the quadrilaterals $U_{n,k}$.

We now define each quadrilateral $U_{n,k}$ by picking its four vertices. The two top ones will be on $\phi_n$, and the two bottom ones will be on $\phi_{n+1}$. Consult Figure \ref{fig:domainside} for a visual of this construction. In fact, we pick the bottom vertices to be the points on $\phi_{n+1}$ whose projections to the boundary on the domain side are equal to the endpoints of the interval $I_{n,k}$. If the leftmost one is called $u_{n,k}^-$, then the rightmost one is naturally $u_{n,k+1}^-$. Let us next explain how to pick the leftmost top vertex $u_{n,k}^+$ on $\phi_n$, and the rightmost one $u_{n,k+1}^+$ will be done by repeating the same construction but for $k+1$ in place of $k$. The choice of $u_{n,k}^+$ will depend on the situation of the top vertex $v_{n,k}^+$ on the target side. There are some cases here.
\begin{itemize}
\item First case, if $v_{n,k}^+$ was a \emph{top point in the middle}. In this case we may simply pick $u_{n,k}^+$ to be the point with the same projection to the boundary as $u_{n,k}^-$, i.e. situated directly above it on $\phi_n$.
\item Second case, here we must refer to how $v_{n,k}^+$ is positioned with respect to the corresponding bottom vertex of a set $V_{n-1,j}$ on the previous dyadic level (the options \emph{left, right,} and \emph{equal} mentioned before). If $v_{n,k}^+$ is equal to this bottom vertex $v_{n-1,j}^-$, then we may define $u_{n,k}^+$ as in the first case. In the other two cases, we proceed as in the first case but at the end we shift the position of $u_{n,k}^+$ on $\phi_n$ either to the left or the right (depending on case) by a fixed amount, say $1/8$ of the length of $I_{n,k}$.
\item Initial case. If $n = 4$, which we chose as our starting index for $n$, then there is no previous dyadic level to be considered and we can define $u_{n,k}^+$ as in the first case, i.e. directly above the left endpoint of $I_{n,k}$.
\end{itemize}
To reiterate, the main point of making the choice in the second case is simply to ensure that the points $u_{n,k}^+$ and $u_{n-1,j}^-$ are in the same order on the curve $\phi_n$ as their prospective image points $v_{n,k}^+$ and $v_{n-1,j}^-$ are on $\Phi_n$, which is a natural requirement if we are constructing a homeomorphism.

\begin{figure}
\centering
\includegraphics[scale=0.3]{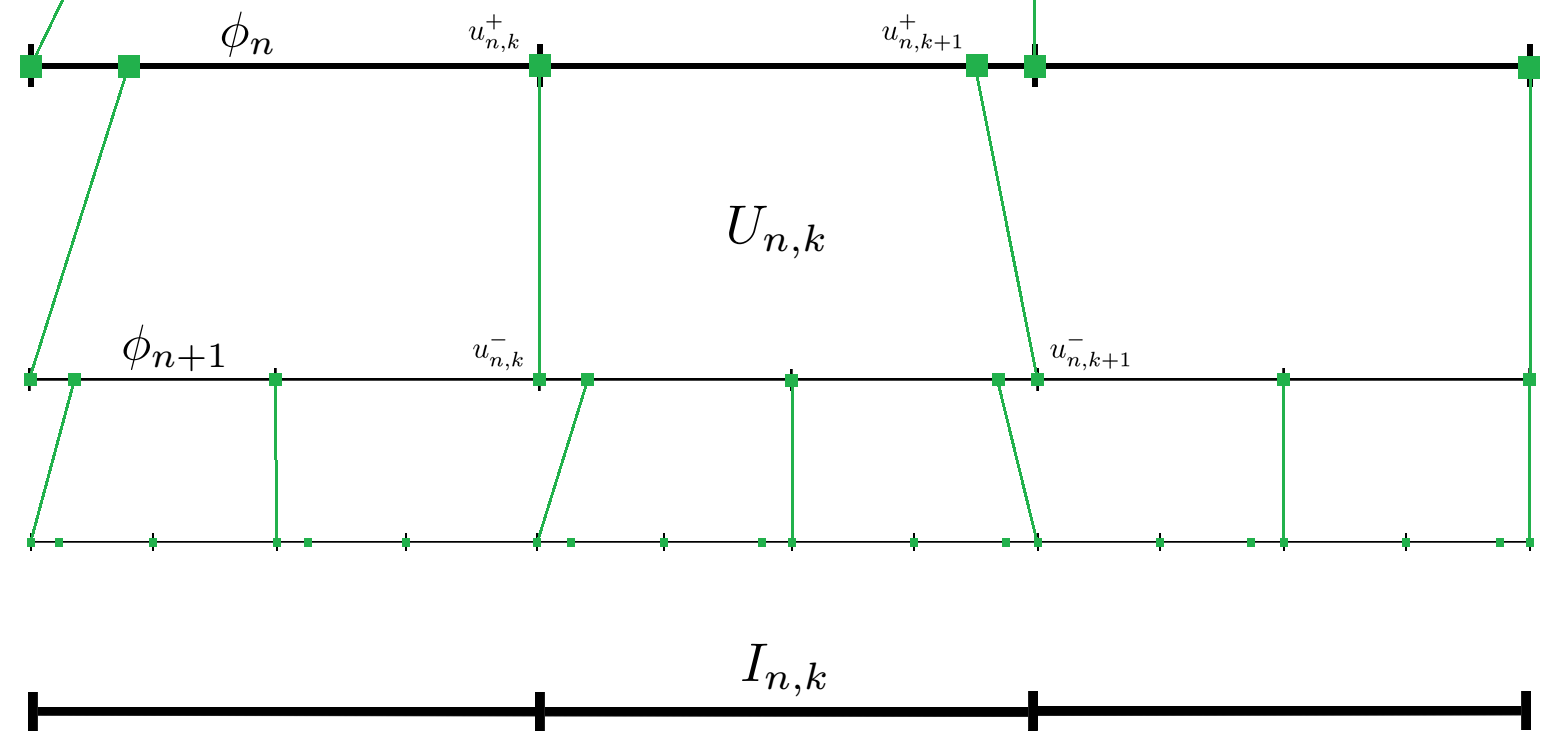}
\caption{The construction of the regions $U_{n,k}$ on the domain side. The boundary here is imagined to be locally flat to simplify the illustration.}
\label{fig:domainside}
\end{figure}

To finish, the quadrilateral $U_{n,k}$ is now defined by its endpoints $u_{n,k}^+$, $u_{n,k}^-$, $u_{n,k+1}^-$, and $u_{n,k+1}^+$. We further note that
\begin{itemize}
\item The quadrilaterals $U_{n,k}$ form a partition of the domain side which is topologically equivalent to the partion of the target side given by the $V_{n,k}$. 
\item The pairwise differences between nearby points are always comparable to $2^{-n}$, where $n$ denotes the dyadic level. This is particularly important for considering differences between bottom points $u_{n,k}^-$ and nearby top points $u_{n+1,j}^+$ of the next dyadic level. The idea here is that by controlling this distance from below, we guarantee that no small part is mapped to a long curve in the target.
\end{itemize}
We may now define the extension map $h : \dd \to \yy$ within the regions $U_{n,k}$. To do this, we simply need to define the map from each of the regions $U_{n,k}$ to the corresponding region $V_{n,k}$ on the target side. This should be done in such a way that the boundary maps between neighbouring regions agree. But as long as this is satisfied and the map $h$ is homeomorphic from $U_{n,k}$ to $V_{n,k}$ for each $n,k$, it will automatically be a homeomorphism and also extend continuously up to the boundary and be equal to $\varphi$ there.

We start by defining the map $h$ on only the collection of all points
\[\mathcal{U} = \{u_{n,k}^+ : n \geq 1, 1 \leq k \leq 2^n\} \cup \{u_{n,k}^- : n \geq 1, 1 \leq k \leq 2^n\}.\] Naturally, for each $n,k$ we will define the map $h$ to send each point $u_{n,k}^+$ to $v_{n,k}^+$ and correspondingly $u_{n,k}^-$ to $v_{n,k}^-$. We next move on to defining $h$ from $\partial U_{n,k}$ to $\partial V_{n,k}$ as a piecewise constant speed map.

The collection $\mathcal{U}$ of points in $\dd$ naturally splits each boundary $\partial U_{n,k}$ into a number segments, the quantity of which will range from four to nine depending on the order of points on $\phi_n$ and $\phi_{n+1}$. By definition $h$ will map the endpoints of such a segment into image points on $\partial V_{n,k}$, and we simply define $h$ on the whole of such a segment as a constant speed map from this segment on to the corresponding part of $\partial V_{n,k}$ between the two image points.

On each fixed quadrilateral $U_{n,k}$, this defines $h$ on the boundary $\partial U_{n,k}$ as a Lipschitz-homeomorphism onto $\partial V_{n,k}$, with Lipschitz-constant controlled by a constant times $|\partial V_{n,k}|/2^{-n}$. We may now use a homeomorphic Lipschitz extension, see e.g. \cite{Kovalev}, to extend this boundary map as a homeomorphic Lipschitz-map from $U_{n,k}$ to $V_{n,k}$ with comparable Lipschitz constant to the boundary map. This defines $h$ within the union of all of the regions $U_{n,k}$.

\begin{figure}[h]\centering
\includegraphics[scale=0.35]{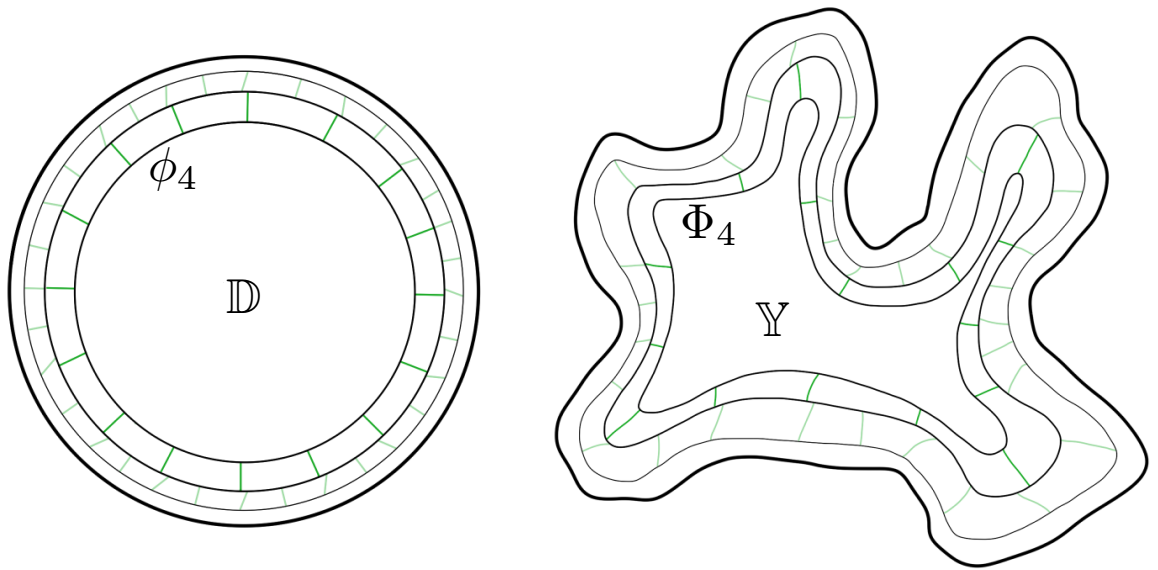}
\caption{A full view of the initial stage of the construction.}\label{fig:fullimg}
\end{figure}

It remains to define $h$ on the central part which is remaining, which is the region bounded by the circle $\phi_4$. Since $h$ has already been defined as a Lipschitz-homeomorphism from the curve $\phi_4$ to $\Phi_4$, we may again apply a homeomorphic Lipschitz extension to define $h$ between the central regions in $\dd$ and $\yy$ bounded by these two curves. See Figure \ref{fig:fullimg}.

Let us conclude by showing that $h$ lies in $\W^{1,p}(\dd)$. On each $U_{n,k}$, the Lipschitz-constant of $h$ will be linearly bounded by the Lipschitz-constant on the boundary $\partial U_{n,k}$. Each segment of such a boundary on which $h$ was defined in constant speed was comparable in length to $2^{-n}$, and the image of such a segment will have length at most a constant times the length of all involved curves on the target side, which are $\Gamma_{n,k}$, $\Gamma^{*}_{n,k}$ and three more similar curves on the dyadic level $n+1$. Let us denote the total length of these five curves by $L_{n,k}$. Thus the Lipschitz-constant of $h$ in $U_{n,k}$ is at most a constant times $L_{n,k}/2^{-n}$, and this yields the Sobolev-estimate
\[\int_{U_{n,k}} |Dh(z)|^p \, dz \leq c_p|L_{n,k}|^p (2^{-n})^{2-p}.\]
Recall now \eqref{eq:partEstimate} and \eqref{eq:newcurve}. Adding up over $n,k$, we may use these estimates along with the fact that the quantity $L_{n,k}$ only included the lengths of five curves close-by to get the desired estimate
\[\int_{\bigcup_{n,k} U_{n,k}} |Dh(z)|^p \, dz \, \leq 5c_p \sum_{n,k} \I_{n,k} \leq \, c_p  \int_{\partial\dd}\int_{\partial\dd} \frac{d_{\yy}(\varphi(x),\varphi(y))^p}{|x-y|^p} \, dx \, dy.\]
In the remaining central region bounded by $\phi_4$, we obtain a similar estimate as the Lipschitz-constant of $h$ in this region is bounded by a constant times the length of $\Phi_4$, which is bounded by the total length of all the curves $\Gamma_{4,k}$ and $\Gamma_{4,k}^*$.

This finishes the construction, as $h$ is now verified to be a homeomorphic $\W^{1,p}$-extension of $\varphi$.
\end{proof}

\section{Maximal operators}

Let \( u \colon \mathbb{R}^n \to \mathbb{R} \) be a locally integrable function. The \emph{Hardy–Littlewood maximal function} of   $u$ is defined by  
\[
\mathbf{M}u(x) = \sup_{r > 0} \dashint_{\mathbb{B}(x, r)} \lvert u(z) \rvert \, \mathrm{d}z \,,
\]
where the integral average is denoted by  $
\dashint_{\mathbb{B}(x, r)} \lvert u \rvert = \lvert \mathbb{B}(x, r) \rvert^{-1} \int_{\mathbb{B}(x, r)} \lvert u \rvert \, \mathrm{d}z  $.
Here, \( \mathbb{B}(x, r) \) is the ball centered at \( x \) with radius \( r > 0 \), and \( \lvert \mathbb{B}(x, r) \rvert \) represents its \( n \)-dimensional Lebesgue measure.

The  classical \emph{Hardy–Littlewood theorem} states that $\mathbf{M} $ is a bounded sublinear operator on  $\L^p(\mathbb{R}^n)$ for $p > 1$. Specifically, there exists a constant $C = C(n, p)$, depending only on the dimension $n$ and the exponent $p$, such that  
\begin{equation}\label{eq:hardy_little}
\lVert \mathbf{M} u \rVert_{\L^p(\mathbb{R}^n)} \leq C \lVert u \rVert_{\L^p(\mathbb{R}^n)} \,.
\end{equation}

Another important maximal operator is the \emph{spherical maximal function}, defined as  
\[
\mathbf{S}u(x) = \sup_{r > 0} \dashint_{\mathbb{S}(x, r)} \lvert u(z) \rvert \, \mathrm{d}z \,,
\]
where \( \mathbb{S}(x, r) = \partial \mathbb{B}(x, r) \) denotes the sphere of radius \( r \) centered at \( x \). The integral average here is computed with respect to the \((n-1)\)-dimensional surface area measure on \( \mathbb{S}(x, r) \).

The spherical maximal function was introduced to harmonic analysis by Stein~\cite{St}. A foundational result in this context, proven by  Stein~\cite{St} for $n \geq 3 $ and Bourgain~\cite{Bo} for $n = 2$, states that $\mathbf{S}$ is bounded on $\L^p(\mathbb{R}^n)$ for \( p > \frac{n}{n-1} \). Specifically, there exists a constant \( C = C(n, p) \) such that  
\begin{equation}\label{eq:bour_stein}
\lVert \mathbf{S}u \rVert_{\L^p(\mathbb{R}^n)} \leq C \lVert u \rVert_{\L^p(\mathbb{R}^n)} \quad \text{for } p > \frac{n}{n-1} \,.
\end{equation}
A more general family of maximal operators, introduced in~\cite{IO} and parameterized by $\theta \geq 1$, is given by
\begin{equation}\label{eq:max_ope}
\mathcal{M}_\theta u(x) = \sup_{r > 0} \left[ \frac{n}{r^n} \int_0^r t^{n-1} \left( \dashint_{\mathbb{S}(x, t)} \lvert u(y) \rvert \, \mathrm{d}y \right)^\theta \mathrm{d}t \right]^{\frac{1}{\theta}} \, .
\end{equation}
This family of operators \( \mathcal{M}_\theta \) interpolates between the Hardy–Littlewood maximal function and the spherical maximal function. Specifically:  
\begin{itemize}
 \item When $ \theta = 1 $, $ \mathcal{M}_\theta $ equals the Hardy–Littlewood maximal function $\mathbf{M}$.  
 \item As \( \theta \to \infty \), \( \mathcal{M}_\theta \) converges to the spherical maximal function $\mathbf{S}$.  
\end{itemize}
As a consequence of the maximal inequalities in~\eqref{eq:hardy_little}  and~\eqref{eq:bour_stein}, it was shown in~\cite{IO}, that there exists a constant $C=C(\theta , p , n)$ such that
\begin{equation}\label{eq:iwaniec_onninen}
\Vert \mathcal{M}_\theta u \rVert_{\L^p(\mathbb{R}^n)} \leq C \lVert u \rVert_{\L^p(\mathbb{R}^n)} \qquad \text{for} \quad  p > \frac{n}{n-1+\frac{1}{\theta}} \,.
\end{equation}

\section{Proof of Theorem \ref{thm:main}, necessity}

We next move onto showing that the condition \eqref{eq:internalDouglas} in our main theorem is necessary, i.e. that it is satisfied by any $\W^{1,p}$-homeomorphism between $\dd$ and $\yy$.

\begin{proof} Let $h : \bar{\dd} \to \bar{\yy}$ be a $\W^{1,p}$-homeomorphism onto the Jordan domain $\yy$. We again use the dyadic decomposition of boundary arcs $I_{n,k}$ to decompose $\partial \dd$. This time we use the dyadic decomposition in order to split the double integral in the internal Douglas condition as follows
\begin{align*}
\numberthis \label{eq:twosums}\int_{\partial \dd}\int_{\partial \dd} &\frac{d_{\yy}(\varphi(x),\varphi(y))^p}{|x-y|^p} \, dx \, dy \\&= 2\sum_{n=1}^\infty \sum_{k=2}^{2^n-1} \int_{I_{n,k-1}}\int_{I_{n,k+1}} \frac{d_{\yy}(\varphi(x),\varphi(y))^p}{|x-y|^p} \, dx \, dy \\&+ 2\sum_{n=1}^\infty \sum_{k=2}^{2^n-2} \int_{I_{n,k-1}}\int_{I_{n,k+2}} \frac{d_{\yy}(\varphi(x),\varphi(y))^p}{|x-y|^p} \, dx \, dy.
\end{align*}
The justification for this equality comes from the fact that, disregarding common boundaries, the sets of the type $I_{n,k-1} \times I_{n,k+1}$ and $I_{n,k-1} \times I_{n,k+2}$ are disjoint and cover the set $\{(x,y) \in \partial \dd \times \partial \dd : x < y\}$. This is perhaps best understood visually, and this fact is illustrated in Figure \ref{fig:box}.

\begin{figure}[h]\centering
\includegraphics[scale=0.75]{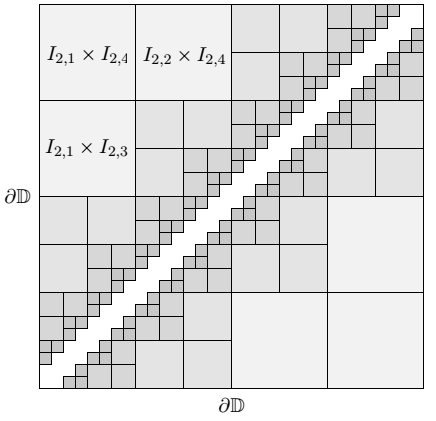}
\caption{Decomposing $\partial \dd \times \partial \dd$ into products of dyadic intervals.}\label{fig:box}
\end{figure}

The basic idea now is to apply some geometric reasoning for each of these pairs of intervals. We will proceed to find an estimate for the first sum on the right hand side from above by a quantity depending on $||Dh||_{\L^p}$. The treatment for the second sum will be essentially the same as for the first one, so we postpone it for now.

We must hence find a way to estimate a single term
\[\int_{I_{n,k-1}}\int_{I_{n,k+1}} \frac{d_{\yy}(\varphi(x),\varphi(y))^p}{|x-y|^p} \, dx \, dy.\]
Let us first interpret the area of integration in a slightly different way. We switch to a coordinate system where we vary a center $m$ over the interval $I_{n,k+1}$ and replace $x,y$ by $m-r$ and $m+r$, where $r$ ranges from $2^{-n-1}$ to $2^{-n+1}$. This is a linear change of variables into a slightly larger domain of integration, so we find the estimate
\begin{align}\label{eq:pairinterval}
\int_{I_{n,k-1}}\int_{I_{n,k+1}} &\frac{d_{\yy}(\varphi(x),\varphi(y))^p}{|x-y|^p} \, dx \, dy \\&\notag \leq C \int_{2^{-n-1}}^{2^{-n+1}}\int_{I_{n,k}} \frac{d_{\yy}(\varphi(m-r),\varphi(m+r))^p}{|2r|^p} \, dm \, dr.
\end{align}

We next need to make a geometrical consideration, see Figure \ref{fig:spherical}. Let us define the dyadic region $G_{n,k} := \{s e^{i\theta} : e^{i\theta} \in I_{n,k},\, 2^{-n-1} \leq 1-s \leq 2^{-n}\}$, and $S_{z,r}$ denotes a circle of radius $r$ centered at $z$. Fix a radius $s \in [1-2^{-n},1-2^{-n-1}]$, so that $sm$ is a point in $G_{n,k}$ which lies on the ray between $m$ and the origin. We note that there is a circle with center $sm$ passing through both $m-r$ and $m+r$, and the radius $R^* = R^*(m,r,s)$ of such a circle is between $r$ and $8r$. Moreover, the transformation $r \mapsto R^*$ is bilipschitz with a constant that can be chosen to be uniform in terms of $k,n,m,$ and $s$. This observation will be required for the forthcoming estimate.

\begin{figure}[h]\centering
\includegraphics[scale=0.2]{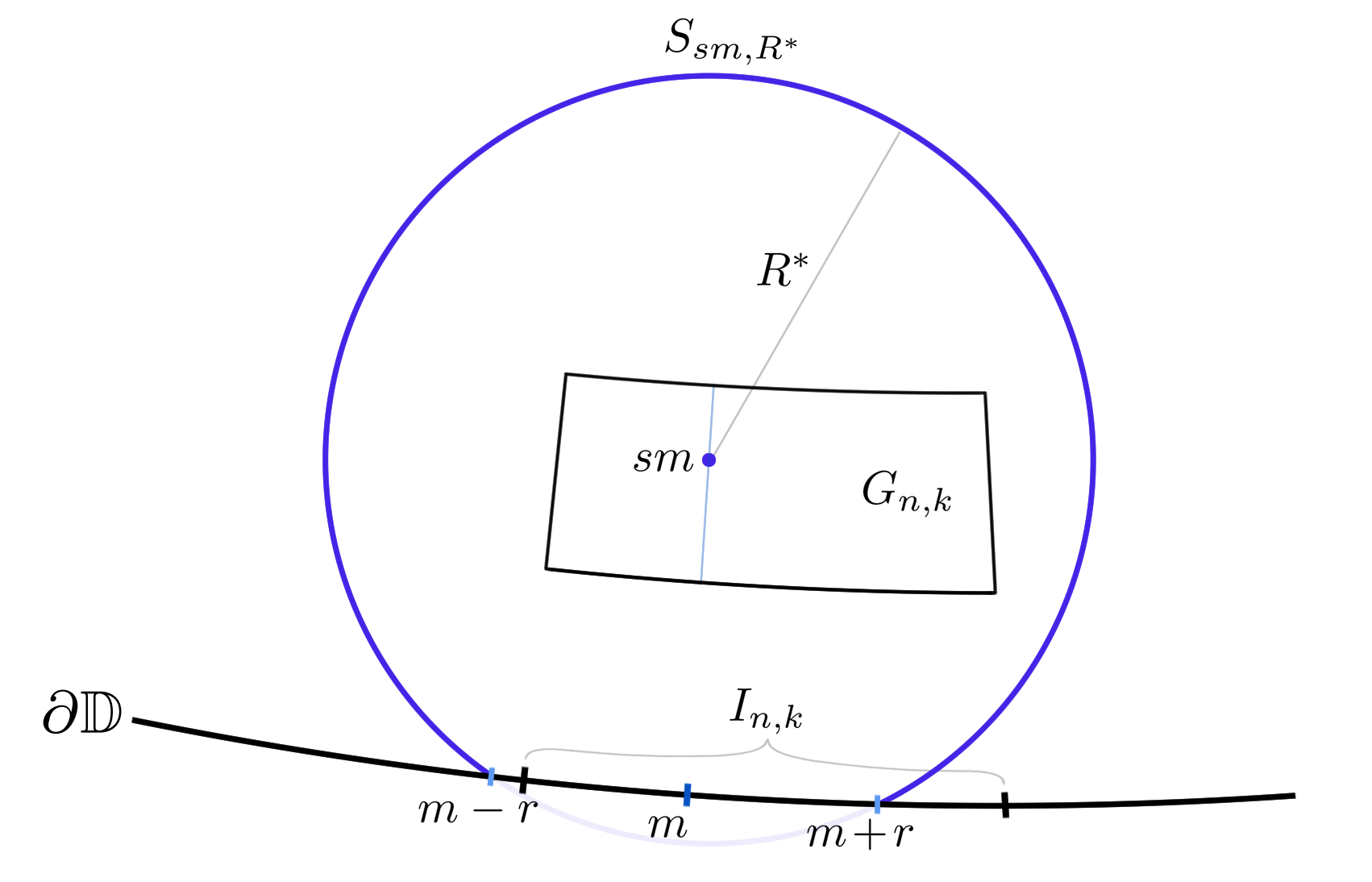}
\caption{Illustration of the geometric construction used to apply the spherical maximal function.}\label{fig:spherical}
\end{figure}

Returning to \eqref{eq:pairinterval}, we note that the internal distance expression $d_{\yy}(\varphi(m-r),\varphi(m+r))$ in the innermost integral is controlled from above by the length of the image curve of $S_{sm,R^*}$ under $h$, giving
\[\frac{d_{\yy}(\varphi(m-r),\varphi(m+r))^p}{|2r|^p} \leq \frac{|h(S_{sm,R^*})|^p}{|2r|^p} \leq C \left(\dashint_{S_{sm,R^*}} |Dh| \right)^p.\]
Note that $h$ is not actually defined on the whole circle $S_{sm,R^*}$, as part of it lies outside of $\dd$, so we interpret $h(S_{sm,R^*})$ as the image of the part of $S_{sm,R^*}$ which lies inside $\dd$. Similarly we interpret $|Dh|$ as $0$ outside of $\dd$. Continuing at \eqref{eq:pairinterval}, we find that
\begin{align*}
\int_{I_{n,k-1}}\int_{I_{n,k+1}} &\frac{d_{\yy}(\varphi(x),\varphi(y))^p}{|x-y|^p} \, dx \, dy \\&\leq
C \int_{2^{-n-1}}^{2^{-n+1}}\int_{I_{n,k}} \left(\dashint_{S_{sm,R^*}} |Dh| \right)^p \, dm \, dr
\\&\leq C \int_{I_{n,k}} \frac{1}{2^{-n}}\int_{2^{-n-1}}^{2^{-n+1}} r \left(\dashint_{S_{sm,R^*}} |Dh| \right)^p \, dr \, dm
\\&\overset{\ast}\leq C\, 2^{-n} \int_{I_{n,k}} \frac{2}{2^{-2n+4}}\int_{0}^{2^{-n+2}} R^* \left(\dashint_{S_{sm,R^*}} |Dh| \right)^p \, dR^* \, dm
\\&\leq C\, 2^{-n} \int_{I_{n,k}} \left(\M_p |Dh|(sm)\right)^p\, dm.
\end{align*}
Here at $\ast$ we have used the bilipschitz-equivalence of $r$ and $R^*$ to change the variable of integration. Since the above chain of estimates is independent of the choice of $s$, we may take an average of $s$ over the interval $[1-2^{-n},1-2^{-n-1}]$ to get
\begin{align*}
\int_{I_{n,k-1}}\int_{I_{n,k+1}}& \frac{d_{\yy}(\varphi(x),\varphi(y))^p}{|x-y|^p} \, dx \, dy \\&\leq C\, \dashint_{1-2^{-n}}^{1-2^{-n-1}}2^{-n}\int_{I_{n,k}} \left(\M_p |Dh|(s m)\right)^p \, dm\, ds
\\&\leq C \int_{1-2^{-n}}^{1-2^{-n-1}}\int_{I_{n,k}} \left(\M_p |Dh|(s m)\right)^p \, dm\, ds
\\&\leq C \int_{G_{n,k}} \left(\M_p |Dh|(z)\right)^p \, dz.
\end{align*}
Here the planar maximal operator $\M_p$ is defined by the formula given in~\eqref{eq:max_ope}.

As the $G_{n,k}$ are disjoint subsets of $\dd$, we may sum over $n$ and $k$ to finally obtain that
\begin{align*}\sum_{n,k}\, \int_{I_{n,k-1}}\int_{I_{n,k+1}} \frac{d_{\yy}(\varphi(x),\varphi(y))^p}{|x-y|^p} \, dx \, dy &\leq C \int_{\dd} \left(\M_p |Dh|(z)\right)^p \, dz \\&\leq C\|\M_p\|_{\L^p (\R^2)}^p \\&\leq  C \|Dh\|_{\L^p (\mathbb R^2)}^p\\&  = C \|Dh\|_{\L^p (\mathbb D)}^p.\end{align*}
The last inequality follows from~\eqref{eq:iwaniec_onninen} since $p > \frac{2}{1+\frac{1}{p}}$.
This finishes our estimates for the first sum in \eqref{eq:twosums}. We explain briefly what modifications to this argument need to be made to also treat the second sum.

In the second sum, the terms to be estimated are of the form
\[\int_{I_{n,k-1}}\int_{I_{n,k+2}} \frac{d_{\yy}(\varphi(x),\varphi(y))^p}{|x-y|^p} \, dx \, dy.\]
The only difference to the first sum is that the integration domain $I_{n,k+2}$ in the inner integral is one dyadic step further away from the interval $I_{n,k}$. We may proceed with the estimate as in the previous case, but simply increase the area of definition of the variable $r$ to $[2^{-n-1},2^{-n+2}]$ instead, having doubled the upper limit. This has no major consequences to the rest of the arguments, as we may still do the same geometric consideration which culminates in an estimate of $\M_p|Dh|$ over the dyadic region $G_{n,k}$. The only difference in the estimates is an additional constant factor coming from the fact that the domain of definition of $R^*$ is also increased.

Hence both of the terms in \eqref{eq:twosums} have been shown to be finite, which finishes the proof.
\end{proof}

\section{Proof of Theorem \ref{thm:rectifiable_p>2}}
In this section, we prove Theorem \ref{thm:rectifiable_p>2} by constructing, for each $p > 2$, an example of a rectifiable Jordan domain $\yy $ and a boundary map $\varphi \colon \partial \dd \to \partial \yy $ that satisfies the $p$-Douglas condition but does not admit a homeomorphic $\W^{1,p}$-extension. However, before presenting the proof, we establish that such an example cannot exist if the boundary of  $\yy $ is piecewise smooth.

\begin{remark}\label{rm:fail_proof}
Let us give a direct proof of the following, perhaps surprising, consequence of Theorem \ref{thm:main}: If $\yy$ has  a piecewise smooth boundary, $p > 2$, and $\varphi : \partial \dd \to \partial \yy$ is a homeomorphic boundary map, then $\varphi$ satisfies the $p$-Douglas condition if and only if it satisfies the internal $p$-Douglas condition. The perhaps unexpected part of this result is the fact that the internal distance on $\yy$ may not be comparable to the Euclidean distance at all, yet there is still a correspondence between the associated conditions.
\\\\
\noindent \emph{Proof or Remark~\ref{rm:fail_proof}.} It is enough to show that the internal $p$-Douglas condition follows from the classical one. The crucial reason why this is true, and why $p > 2$ is needed, is given by the result of Lemma 2.1 in \cite{KOext3}. Essentially, this lemma allows us to conclude that if $p > 2$, then in the $p$-Douglas condition it is possible to omit pairs $x,y$ whose images under $\varphi$ belong to different smooth pieces on the boundary, without weakening the condition. This in turn allows us to argue as follows.

We wish to estimate the double integral in \eqref{eq:internalDouglas} from above. First of all, it is enough to restrict the integral to pairs of points $x,y \in \partial \dd$ which are close together, since if $|x - y|$ is bounded uniformly below, the denominator can be estimated away. Moreover, there is no need to consider points $x,y$ whose images under $\varphi$ are on the same smooth piece on $\partial \yy$, as the internal distance between such points in $\yy$ is comparable to the Euclidean one.

Hence we consider a case where $\varphi(x) \in \Gamma_1$ and $\varphi(y) \in \Gamma_2$, where $\Gamma_1$ and $\Gamma_2$ are neighboring smooth pieces of $\partial \yy$. Suppose that the point at which these pieces meet is $\varphi(z)$, with $z$ between $x$ and $y$ on $\partial \dd$. Then
\[
\begin{split}
d_{\yy}(\varphi(x),\varphi(y)) & \leq d_{\yy}(\varphi(x),\varphi(z))  + d_{\yy}(\varphi(z),\varphi(y)) \\ & \leq C \left(|\varphi(x)-\varphi(z)| + |\varphi(z) - \varphi(y)|\right)
\end{split}
\]
for an appropriate constant $C$ depending on the geometry of $\yy$. One can now conclude by estimating the double integral as in the second chain of inequalities on page 10 in \cite{KOext3}.
\end{remark}

We now continue to providing the counterexample of Theorem \ref{thm:rectifiable_p>2}.

\begin{proof}[Proof of Theorem \ref{thm:rectifiable_p>2}]
Fix $p > 2$, and let us begin the construction of the domain $\yy$ with a curve $\Gamma$ whose endpoints will be at $(0,1)$ and $(1,0)$. The curve will stay within the area bounded by the two rays $R_\pm = \{(1-x,\pm x) : x > 0\}$. Intuitively, $\Gamma$ will consist of line segments which zigzag between these two rays.

Moreover, this curve will be parametrized by a map $g : [0,1] \to \Gamma$, and we will now define both $\Gamma$ and $g$. We start by splitting the interval $[0,1]$ into countably many subintervals $I_n = [A_n,A_{n+1}]$, with $(A_n)$ being the decreasing sequence defined by $A_n = n^{-p/(p-2)}$.

Each $I_n$ will be mapped linearly to a line segment $J_n \subset \Gamma$ by $g$, and $\Gamma$ is therefore defined as the union of such line segments. The line segments $J_n$ will have one endpoint on $R_+$ and the other one on $R_-$, with the order alternating with the parity of $n$. In fact, we may make the choice of the segments $J_n$ so that their lengths are comparable to any particular choice of sequence $(b_n)$ of positive numbers decreasing to zero. For this construction we pick $b_n := n^{-2}$, so our segments $J_n$ are thus defined by the requirement $|J_n| \approx n^{-2}$.

Let us also denote $a_n := A_{n+1} - A_{n}$, so that $a_n \approx n^{-2(p-1)/(p-2)}$, and we also let $B_n := \sum_{k=n}^\infty b_n \approx \frac{1}{n}$.
\begin{figure}[h]
\centering
\includegraphics[scale=0.3]{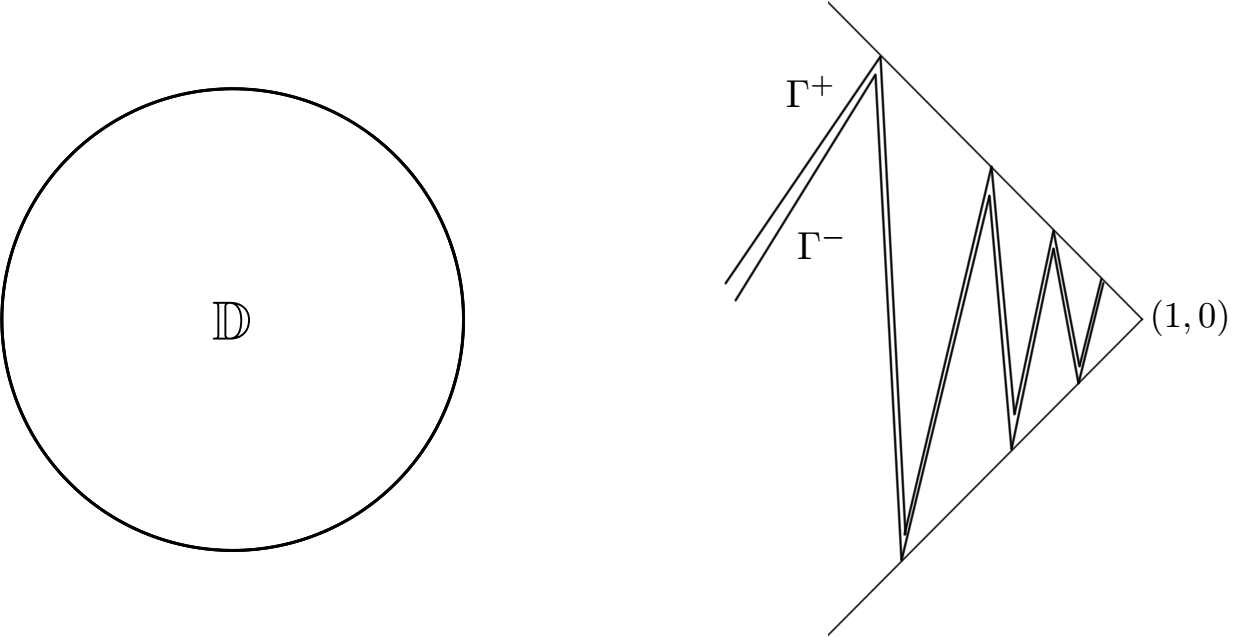}
\caption{The rectifiable target curve $\partial\yy$ defined via curves $\Gamma_\pm$ converging to $(1,0)$.}\label{fig:rectifiable}
\end{figure}

We next define a Jordan curve $\partial\yy$ by fattening the curve $\Gamma$ as in Figure \ref{fig:rectifiable}. In short, each segment $J_n \subset \Gamma$ is replaced by two segments $J_n^\pm \subset \partial \yy$ which are obtained from $J_n$ by moving one endpoint slightly up or down. The endpoints of these new segments are at distance $\epsilon_n$ away from the endpoints of $J_n$, where $\epsilon_n$ is some rapidly decreasing sequence (say, $\epsilon_n := 10^{-100n}$). The new curve $\partial \yy$ will consist of three parts: One small initial segment $\Gamma_0$ near $(0,1)$, and two curves, $\Gamma^+$ and $\Gamma^-$, obtained as the union of the segments $J_n^+$ and $J_n^-$ respectively, and converging at the point $(1,0)$.

On the domain side, without loss of generality we replace the unit disk $\dd$ with an equilateral triangle $T$ with side length one. Let us also pick two sides of $T$ and denote them by $T^\pm \subset \partial T$. To define the boundary map, we reinterpret the parametrization $g : [0,1] \to \Gamma$ as two parametrizations $g^+ : T^+ \to \Gamma^+$ and $g^- : T^- \to \Gamma^-$ in the obvious way. We define the boundary map $\varphi : \partial T \to \partial \yy$ by combining the maps $g^{\pm}$ and mapping the remaining side of the triangle $T$ into the small initial segment $\Gamma_0$ linearly.

We now claim that $\varphi$ gives the required counterexample.

\begin{center}
Step 1: $\varphi$ has a Sobolev extension
\end{center}

Let us first verify that $\varphi$ has a $\W^{1,p}$-extension. For this we simply need to check that $\varphi$ lies in the trace-space of $\W^{1,p}$, equivalently that it satisfies the $p$-Douglas condition
\[\int_{\partial T}\int_{\partial T} \frac{|\varphi(x)-\varphi(y)|^p}{|x-y|^p}\, dx \, dy < \infty.\]
We calculate this integral by splitting it into parts depending on where the points $x,y \in \partial T$ lie. In fact, since the parts $T_+$ and $T_-$ are mapped close together under $\varphi$, it is enough to calculate the double integral over the subdomain $T_+ \times T_+$, as the integral over the other relevant parts is either comparable or strictly smaller than over this part.

We abuse notation here and identify $T_+$ with $[0,1]$, in particular we consider the intervals $I_n = [A_n,A_{n+1}]$ and split the integral further into parts where $x \in I_n, y \in I_m$, and we may also assume $m \geq n$.

\textbf{Case 1.} If $n=m$. Since the map $\varphi$ is linear on each line segment $I_n$, we obtain that
\[\frac{|\varphi(x)-\varphi(y)|}{|x-y|} = \frac{|J_n|}{|I_n|} = \frac{b_n}{a_n} = n^{\frac{2}{p-2}}.\]
This gives that
\begin{align*}
\mathcal{I}_{n,m} := \int_{I_m}\int_{I_n} \frac{|\varphi(x)-\varphi(y)|^p}{|x-y|^p}\, dx \, dy &\leq n^{-4\frac{p-1}{p-2}}n^{\frac{2p}{p-2}} = \frac{1}{n^2}.
\end{align*}
This is summable over $n$, so this case only contributes a finite number to the total integral.

\textbf{Case 2.} If $n+1=m$. This is handled very similarly as the previous case, so we omit the details.

\textbf{Case 3.} If $m > n+1$. Now $|x-y|$ is bounded from below by a positive number $d := A_{n+1} - A_m$ denoting the distance between the two intervals. Moreover, $\varphi(x)$ and $\varphi(y)$ can be considered as points within a triangle that is bounded by the rays $R_+,R_-$ and the segment $J_n$, which has diameter comparable to $|J_n| = b_n = 1/n^2$. Hence $|\varphi(x) - \varphi(y)| \leq 1/n^2$. We therefore obtain that
\begin{align*}
\int_{I_m}\int_{I_n} \frac{|\varphi(x)-\varphi(y)|^p}{|x-y|^p}\, dx \, dy &\leq |I_n| |I_m| d^{-p}n^{-2p} = m^{-2\frac{p-1}{p-2}}n^{-2\frac{p^2 - p - 1}{p-2}}d^{-p}.
\end{align*}
We then write
\begin{equation}\label{eq:nmcalc}d \approx A_n - A_m = \frac{1}{n^{\frac{p}{p-2}}} - \frac{1}{m^{\frac{p}{p-2}}} = m^{-\frac{p}{p-2}}n^{-\frac{p}{p-2}}\left(m^{\frac{p}{p-2}}-n^{\frac{p}{p-2}}\right).\end{equation}
Note that for $x,\alpha > 1$ we have that
\[\frac{x^\alpha - 1}{x - 1} \geq C_\alpha x^{\alpha-1}\]
Therefore with the choice $\alpha = \frac{p}{p-2}$ we have
\[\frac{m^{\frac{p}{p-2}}-n^{\frac{p}{p-2}}}{m-n} \geq C_p m^{\frac{2}{p-2}}.\]
This estimate gives
\begin{align*}\mathcal{I}_{n,m} &\leq m^{-2\frac{p-1}{p-2}}n^{-2\frac{p^2 - p - 1}{p-2}} m^{\frac{p^2}{p-2}}n^{\frac{p^2}{p-2}}m^{\frac{-2p}{p-2}} (m-n)^{-p} \\&= m^{\frac{p^2-4p+2}{p-2}} n^{\frac{-p^2+2p+2}{p-2}}(m-n)^{-p}\end{align*}
We want to sum this over $m$, but since the exponent $s:= \frac{p^2-4p+2}{p-2}$ on $m$ here may be negative we split into two cases.

\emph{Case 3A.} If $s \geq 0$. Then we use the estimate $m^s \leq C((m-n)^s + n^s)$ to find out that
\begin{align*}\sum_{m > n+1}\frac{m^{\frac{p^2-4p+2}{p-2}}}{(m-n)^{p}} &\leq C\, \sum_{k > 1}\frac{k^{\frac{p^2-4p+2}{p-2}}}{k^{p}} + C \,n^{\frac{p^2-4p+2}{p-2}} \sum_{k > 1}\frac{1}{k^{p}} \\&\leq C_1 + C_2 \,n^{\frac{p^2-4p+2}{p-2}} \\&\leq C_3 \, n^{\frac{p^2-4p+2}{p-2}},\end{align*}
where we have substituted $k = m-n$, and the first series in $k$ converges due to the inequality $s - p < -1$.

\emph{Case 3B.} If $s < 0$. Then we do the basic estimate $m \geq n$ and obtain
\[\sum_{m > n+1}\frac{m^{\frac{p^2-4p+2}{p-2}}}{(m-n)^{p}} \leq n^{\frac{p^2-4p+2}{p-2}}\sum_{m > n+1}\frac{1}{(m-n)^{p}} \leq C \,n^{\frac{p^2-4p+2}{p-2}},\]
giving the same conclusion as Case 3A.

In other words, we now find that
\[\sum_{m > n+1} \mathcal{I}_{n,m} = n^{\frac{-p^2+2p+2}{p-2}}\sum_{m > n+1}\frac{m^{\frac{p^2-4p+2}{p-2}}}{(m-n)^{p}} \leq C n^{\frac{-p^2+2p+2}{p-2}}n^{\frac{p^2-4p+2}{p-2}} = n^{-2}.\]
This is summable over $n$, so in conclusion $\sum_n \sum_{m \geq n} \mathcal{I}_{n,m}$ is finite, which gives the desired result.

\begin{center}
Step 2: $\varphi$ has no homeomorphic Sobolev-extension
\end{center}

We apply Theorem \ref{thm:main} to attempt to conclude that the map $\varphi$ defined before has no homeomorphic $\W^{1,p}$-extension. Due to the bilipschitz equivalence between the unit disk and the triangle $T$, it will be enough to verify that
\begin{equation}\label{eq:internalDouglas2}
\int_{\partial T}\int_{\partial T} \frac{[d_{\yy}(\varphi(x),\varphi(y)]^p}{|x-y|^p} \, dx \, dy \ = \ \infty.
\end{equation}
In fact, we will show that the integral diverges even when restricted to a smaller domain, which will be a subset of $T^+ \times T^+$. For this purpose, we identify the side $T^+$ with the unit interval $[0,1]$, and pick points $x \in I_n$ and $y \in I_m$ for some indices $n,m$ satisfying $m \geq 2n$.

Given the way $\partial \yy$ was constructed, the internal distance between two points on the same part of the boundary, say $\Gamma^+$, is comparable to the length of the shorter boundary arc between them. Hence by our choices of $x$ and $y$, the internal distance of $\varphi(x)$ and $\varphi(y)$ may be estimated by
\[d_\yy(\varphi(x),\varphi(y)) \approx \sum_{k=n}^{m} |J_k| = \sum_{k=n}^{m} \frac{1}{k^2} \approx \frac{1}{n} - \frac{1}{m} \approx \frac{1}{n},\]
where we applied the fact that $m \geq 2n$. As in \eqref{eq:nmcalc}, we find that $|x-y| \approx A_{n+1} - A_m \approx n^{-\frac{p}{p-2}}$. This lets us compute that
\begin{align*}
\int_{I_m} \int_{I_n} \frac{[d_{\yy}(\varphi(x),\varphi(y)]^p}{|x-y|^p} \, dx \, dy &\geq c\int_{I_m} \int_{I_n} \frac{n^{-p}}{n^{-\frac{p^2}{p-2}}} \, dx \, dy
\\&= c\,m^{-2\frac{p-1}{p-2}}\, n^{-2\frac{p-1}{p-2} - p + \frac{p^2}{p-2}}
\\&= c\,m^{-2\frac{p-1}{p-2}}\, n^{\frac{2}{p-2}}
\end{align*}
Summing this over $m \geq 2n$ gives
\[\sum_{m=2n}^\infty \int_{I_m} \int_{I_n} \frac{[d_{\yy}(\varphi(x),\varphi(y)]^p}{|x-y|^p} \, dx \, dy \geq c\, n^{\frac{2}{p-2}} \, (2n)^{1-2\frac{p-1}{p-2}} = c_1\, n^{-1}.\]
The right hand side is not summable over $n$, which proves our claim as the sets $I_n \times I_m$ constitute disjoint subsets of $\partial T \times \partial T$.
\end{proof}

\section{Proof of Theorem \ref{thm:p=1case}}

We provide here an example of a Jordan domain $\yy$ and a homeomorphism $h : \bar{\dd} \to \bar{\yy}$ which lies in $\W^{1,1}(\dd)$, but whose restriction to the boundary does not satisfy the internal Douglas condition \eqref{eq:internalDouglas} for $p = 1$.

\begin{proof}
Let us construct $\varphi$ and the domain $\yy$ at the same time. We first pick a center point, say $(1,0)$, on $\partial \dd$. We next define a sequence of boundary arcs $I_n \subset \partial \dd$ from the upper half part of the unit circle such that $|I_n| = 2^{-n}$, the intervals $I_n$ are mutually disjoint except for $I_n$ and $I_{n+1}$ sharing a common endpoint, and so that $I_n$ approaches the point $(1,0)$ as $n \to \infty$. We define another such series of intervals $I_n^*$ by reflecting these intervals across the $x$-axis.

We now define the main part of the boundary $\partial \yy$ by mapping each $I_n$ and $I_n^*$ via the map $\varphi$ in constant speed to target curves $J_n$ and $J_n^*$ which we now define. These image curves are chosen to be piecewise linear curves zigzagging within a conical region of the plane just as in the construction of the previous section. See Figure \ref{fig:zigzag2} for a visual. The only important restriction we place is that we have the length estimate
\[|J_n| \approx |J_n^*| \approx \frac{2^n}{n^2},\]
where $\approx$ means comparable via universal constant. We also place the curves $J_n$ and $J_n^*$ very close together, so that the internal distance between two points on $J_n$ becomes comparable to the length of the boundary between them.

The rest of the boundary of $\partial \yy$ may be defined via connecting the two initial endpoints of $J_0$ and $J_0^*$ via a line segment, and mapping the remaining part of $\partial \dd$ into this line segment at constant speed via $\varphi$.
\begin{figure}[h]
\centering
\includegraphics[scale=0.25]{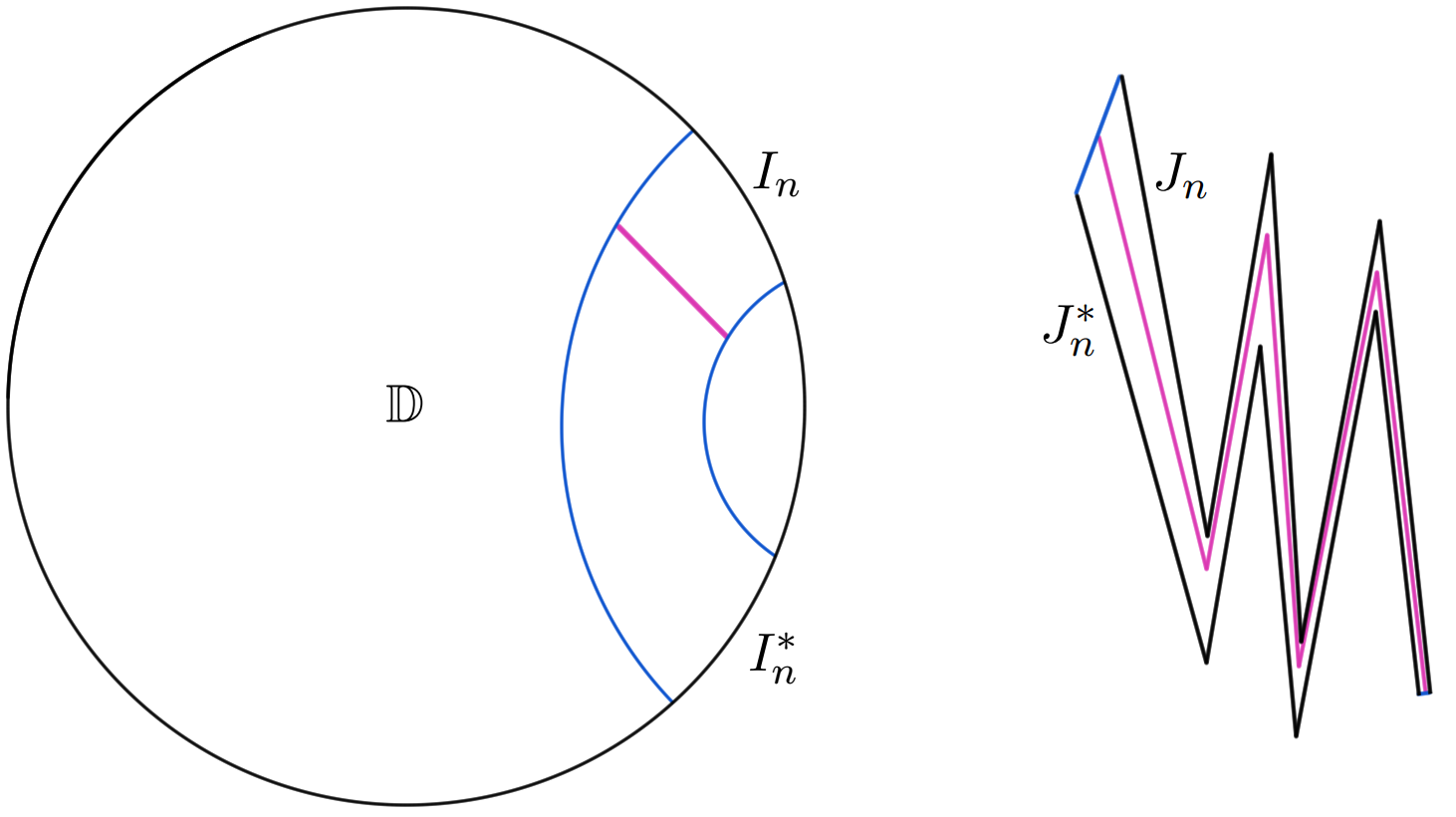}
\caption{The map $h : \dd \to \yy$, with the image curve of one radial segment shown. Note that only a part of the target boundary is pictured.}\label{fig:zigzag2}
\end{figure}

Let us now prove that $\varphi$ has a homeomorphic extension in $\W^{1,1}(\dd)$. The construction is fairly straightforward, so we omit some details here. The main idea is simply to connect the endpoints of $I_n$ to the respective endpoints of $I_n^*$ via circles perpendicular to the boundary in $\partial \dd$, defining an annular region $A_n$ within $\dd$. We may cover such a region $A_n$ with ``radial segments", i.e. curves connecting the two boundary circles, which may be chosen as straight line segments from one circle to another pointing at $(1,0)$, except near $I_n$ and $I_n^*$ at the boundary, where they should be curved slightly to account for the curvature of $\partial \dd$.

The extension $h$ will map each $A_n$ to a corresponding region in $\yy$ which lies between the curves $J_n$ and $J_n^*$. Since the curves $J_n$ and $J_n^*$ lie near each other, it is natural to define the map $h$ in $A_n$ by mapping each radial segment of $A_n$ into a corresponding piecewise linear curve between $J_n$ and $J_n^*$ which follows these curves closely. See Figure \ref{fig:zigzag2}.

In particular, this gives the Sobolev-estimate
\[\int_{A_n} |Dh| \, dz \approx 2^{-n}|J_n| \approx \frac{1}{n^2}.\]
As this quantity is summable over $n$, our map lies in the class $\W^{1,1}(\dd)$.

In order to estimate the double integral in the internal Douglas condition, let us consider the mutually disjoint sets $I_n \times I_m$, with $m \geq n+2$. We wish to estimate the internal distance $d_{\yy}(\varphi(x),\varphi(y))$, where $x \in I_n$ and $y \in I_m$. Since the boundary of the target has been chosen to be the union of sufficiently thin ``tubes", any curve within $\yy$ connecting $\varphi(x)$ with $\varphi(y)$ must have length comparable to the length of the boundary between these two points. Hence
\[d_{\yy}(\varphi(x),\varphi(y)) \approx \sum_{k=n}^m |J_k| \approx \sum_{k=n}^m \frac{2^k}{k^2} \approx \frac{2^m}{m^2}.\]
Moreover, we have that $|x-y| \approx 2^{-n}$. Hence
\[\int_{I_n}\int_{I_m} \frac{d_{\yy}(\varphi(x),\varphi(y))}{|x-y|} dx dy \approx 2^{-n}2^{-m} \frac{2^m}{m^2 2^{-n}} = \frac{1}{m^2}.\]
However, summing over $n$ and $m$, we find that
\[\sum_{n=1}^\infty \sum_{m=n+2}^\infty \frac{1}{m^2} \approx \sum_{n=1}^\infty \frac{1}{n} = \infty.\]
This proves the claim that $\varphi$ does not satisfy the internal Douglas condition.

\end{proof}


\begin{thebibliography}{0}




\bibitem{Anb}
S. S.  Antman, \textit{Nonlinear problems of elasticity. Applied Mathematical Sciences}, 107. Springer-Verlag, New York, 1995.

\bibitem{AIMb}
K. Astala, T. Iwaniec, and G. J. Martin, \textit{Elliptic partial differential equations and quasiconformal mappings in the plane}, Princeton University Press, 2009.





\bibitem{Bac}
J. M. Ball, \textit{Convexity conditions and existence theorems in nonlinear elasticity},  Arch. Rational Mech. Anal. {63} (1976/77), no. 4, 337--403.

\bibitem{Ba2}
J. M. Ball, \textit{Progress and Puzzles in Nonlinear Elasticity}, Proceedings of course on Poly-, Quasi- and Rank-One Convexity in Applied Mechanics, CISM, Udine, (2010).




\bibitem{Bo}
J. Bourgain, {\it Estimations de certaines fonctions maximales. (French) [Estimates of some
maximal operators]} C. R. Acad. Sci. Paris Sr. I Math. 301 (10) (1985), 499--502



\bibitem{Cib}
P. G. Ciarlet, \textit{Mathematical elasticity Vol. I. Three-dimensional elasticity}, Studies in Mathematics and its Applications, 20. North-Holland Publishing Co., Amsterdam, 1988.

\bibitem{Do}
J. Douglas,   {\em Solution of the problem of Plateau},  Trans. Amer. Math. Soc. { 33} (1931) 231--321.


\bibitem{Dub}
P. Duren, \textit{Harmonic mappings in the plane},  Cambridge University Press, Cambridge, (2004).

\bibitem{Ga}
E. Gagliardo,  {\em Caratterizzazioni delle tracce sulla frontiera relative ad alcune classi di funzioni in n variabili},   Rend. Sem. Mat. Univ. Padova 27 (1957), 284--305.


\bibitem{GeMo}
F. W. Gehring, \textit{Characteristic properties of quasidisks},
S\'eminaire de Math\'ematiques Sup\'erieures [Seminar on Higher Mathematics], 84. Presses de l Universit\'e de Montr\'eal, Montreal, Que., (1982).

\bibitem{Ge2}
F. W. Gehring,  \textit{Uniform domains and the ubiquitous quasidisk},  Jahresber. Deutsch. Math.-Verein. {89}, (1987), 88--103.


\bibitem{HKOext}
S. Hencl,   A. Koski, and  J. Onninen,  \textit{Sobolev homeomorphic extensions from two to three dimensions}, J. Funct. Anal. 286 (2024), no. 9, Paper No. 110371, 51 pp.


\bibitem{IMb}
T. Iwaniec and G. Martin,  \textit{Geometric Function Theory and Non-linear Analysis}, Oxford Mathematical Monographs, Oxford University Press, (2001).

\bibitem{IMOsur2} T. Iwaniec, G. Martin, and J. Onninen \textit{Energy-minimal Principles in Geometric Function Theory}, New Zealand J. Math. 52  [2021–2022], 605–642.



\bibitem{IO}
T. Iwaniec  and J. Onninen,  \textit{  $\mathscr H^1$-estimates of Jacobians by subdeterminants},
Math. Ann.   324 (2002), no. 2, 341--358.
\bibitem{IOinnerv}
T. Iwaniec  and J. Onninen,  \textit{Invertibility versus Lagrange equation for traction free energy-minimal deformations} Calc. Var. Partial Differential Equations 52 (2015), no. 3-4, 489--496.






\bibitem{Ki}
 Kirszbraun, M. D.  {\em \"Uber die zusammenziehende und Lipschitzsche Transformationen}, Fund. Math. 22, (1934) 77--108.



\bibitem{KKO}
P.  Koskela, A. Koski, and J. Onninen, \textit{Sobolev homeomorphic extensions onto John domains},  J. Funct. Anal. 279 (2020), no. 10, 108719, 17 pp.



\bibitem{KOext1}
A. Koski  and  J. Onninen, \textit{Sobolev homeomorphic extensions},  J. Eur. Math. Soc. (JEMS) 23 (2021), no. 12, 4065--4089.


\bibitem{KOext3}  
A. Koski  and  J. Onninen, \textit{The Sobolev Jordan-Sch\"onflies Problem}, Adv. Math. 413 (2023), Paper No. 108795, 32 pp.



\bibitem{KObiSobo}
 A. Koski  and  J. Onninen, \textit{Bi-Sobolev extensions}, J. Geom. Anal. 33 (2023), no. 9, No. 301, 18 pp.


\bibitem{Kovalev}
L. V. Kovalev,  \textit{Optimal extension of Lipschitz embeddings in the plane}, Bull. Lond. Math. Soc. 51 (2019), no. 4, 622--632.


\bibitem{NV}
R. N\"akki and J. V\"ais\"al\"a, \textit{John disks},  Expo. Math. 9, (1991), 3--43.





\bibitem{Reb}
Y. G. Reshetnyak, \textit{Space mappings with bounded distortion}, American Mathematical Society, Providence, RI, 1989.


\bibitem{St_sp}
E. M. Stein,  \textit{Maximal functions. I. Spherical means.} Proc. Nat. Acad. Sci. U.S.A. 73 (7)
(1976), 2174--2175.

\bibitem{St}
E. M. Stein, \textit{Singular Integrals and Differentiability Properties of Functions}, Princeton Math.
Ser. 30, Princeton Univ. Press, Princeton, NJ (1970).


\bibitem{Ve}
G. C. Verchota,  {\em Harmonic homeomorphisms of the closed disc to itself need be in $\W^{1,p}, p<2$, but not $\W^{1,2}$}, 
Proc. Amer. Math. Soc. {\bf 135} (2007), no. 3, 891--894.


\bibitem{Va1.5}
J.  V\"ais\"al\"a,  {\em Uniform domains}. - Tohoku Math. J. (2) 40, (1988), 101--118.

\bibitem{Zh}
Y.R.-Y. Zhang, Schoenflies solutions with conformal boundary values may fail to be Sobolev, Ann. Acad. Sci. Fenn. Math. 44 (2019), 791-796

\end{thebibliography}
\end{document}